\documentclass[11pt]{article}
\usepackage[usenames,dvipsnames,svgnames,table]{xcolor} 

\usepackage{graphicx}
\usepackage{epsfig}
\usepackage{amssymb, amsmath}
\usepackage{verbatim}
\usepackage{authblk}
\usepackage{kotex}
\usepackage{multirow}
\usepackage[colorlinks]{hyperref} 
\usepackage[colorinlistoftodos, textsize=scriptsize]{todonotes} 

\newtheorem{theorem}{Theorem}[section]
\newtheorem{lemma}[theorem]{Lemma}
\newtheorem{proposition}[theorem]{Proposition}
\newtheorem{corollary}[theorem]{Corollary}

\newenvironment{proof}[1][Proof]{\begin{trivlist}
\item[\hskip \labelsep {\bfseries #1}]}{\end{trivlist}}
\newenvironment{definition}[1][Definition]{\begin{trivlist}
\item[\hskip \labelsep {\bfseries #1}]}{\end{trivlist}}

\newenvironment{remark}[1][Remark]{\begin{trivlist}
\item[\hskip \labelsep {\bfseries #1}]}{\end{trivlist}}

\newcommand{\bea}{\begin{eqnarray*}}
\newcommand{\eea}{\end{eqnarray*}}
\newcommand{\bean}{\begin{eqnarray}}
\newcommand{\eean}{\end{eqnarray}}
\newcommand{\baln}{\begin{align}}
\newcommand{\ealn}{\end{align}}
\newcommand{\bdefi}{\begin{definition}}
\newcommand{\edefi}{\end{definition}}
\newcommand{\bi}{\begin{itemize}}
\newcommand{\ei}{\end{itemize}}
\newcommand{\benu}{\begin{enumerate}}
\newcommand{\eenu}{\end{enumerate}}

\newcommand{\E}{{\rm E}}
\newcommand{\V}{{\rm Var}}

\newcommand{\sg}{\Sigma}
\newcommand{\Esg}{\bbE_{\sg_0}} 

\newcommand{\what}{\widehat}

\newcommand{\sumij}{\sum_{i=1}^p \sum_{j=1}^p} 
\newcommand{\bfX}{{\bf X}}

\newcommand{\bbE}{\mathbb{E}}

\newcommand{\bbP}{\mathbb{P}} 
\newcommand{\bbR}{\mathbb{R}}
\newcommand{\cC}{\mathcal{C}}

\newcommand{\cP}{\mathcal{P}}
\newcommand{\cU}{\mathcal{U}}
\newcommand{\cX}{\mathcal{X}}

\newcommand{\lra}{\longrightarrow}

\newcommand{\ra}{\rightarrow}

\newcommand{\calC}{\mathcal{C}}
\newcommand{\calD}{\mathcal{D}}

\newcommand{\calL}{\mathcal{L}}

\newcommand{\calR}{\mathcal{R}}

\begin{document}
	
\title{Optimal Bayesian Minimax Rates for Unconstrained Large Covariance Matrices}
\author[1]{Kyoungjae Lee}
\author[2]{Jaeyong Lee}
\affil[1]{Department of Applied and Computational Mathematics and Statistics, The University of Notre Dame}
\affil[2]{Department of Statistics, Seoul National University}

\maketitle

	\begin{abstract}
		We obtain the optimal Bayesian minimax rate for the unconstrained large covariance matrix of multivariate normal sample with mean zero,  when both the sample size, $n$, and the dimension, $p$, of the covariance matrix tend to infinity. 
		Traditionally the posterior convergence rate is used to compare the frequentist asymptotic performance of priors, but defining the optimality with it is elusive. We propose a new decision theoretic framework for prior selection and define {\it Bayesian minimax rate}.
		Under the proposed framework, we obtain the optimal Bayesian minimax rate for the spectral norm for all rates of $p$. We also considered Frobenius norm, Bregman divergence and squared log-determinant loss and obtain the optimal Bayesian minimax rate under certain rate conditions on $p$. A simulation study is conducted to support the theoretical results. 
	\end{abstract}
	
Key words: Bayesian minimax rate; Convergence rate; Decision theoretic prior selection; Unconstrained covariance.

\section{Introduction}
Estimating covariance matrix  plays a fundamental role in multivariate data analysis. Many statistical methods in multivariate data analysis such as the principle component analysis, canonical correlation analysis, linear and quadratic discriminant analysis require the estimated  covariance matrix as the starting point of the analysis. 
In the risk management and the longitudinal data analysis, the covariance matrix estimation is a crucial part of the analysis. 
The log-determinant of covariance matrix is used for constructing hypothesis test or quadratic discriminant analysis \cite{anderson03}.

Suppose we observe a random sample $\bfX_n = (X_1,\ldots, X_n), ~X_i \in \bbR^p,~ i=1,\ldots,n,$ from the $p$-dimensional normal distribution with mean zero and covariance matrix $\Sigma$, i.e.
\bea
X_1,\ldots, X_n \mid \Sigma &\overset{iid}{\sim}& N_p (0, \Sigma ).
\eea
We assume the zero mean and focus on the covariance matrix. 

With advance of technology, data arising from various areas such as climate prediction, image processing, gene association study, and proteomics, are often high dimensional. 
In such high dimensional settings, it is often natural to assume that the dimension of the variable $p$ tends to infinity as the sample size $n$ gets larger, i.e. $p = p_n \lra \infty$ as $n \lra \infty$. 
This assumption can be justified as follows.   
First, when $p$ is large in comparison with $n$, often the limiting scenario with $p$ tending to infinity approximates closer to the reality than that with $p$ fixed. 
Second, in many cases we can postulate the reality is infinitely complex and involves infinitely many variables, and with limited resources and time, we can collect only a portion of variables and observations. If we have more resources to collect more data, it is natural to collect more observations as well as more variables, i.e. to increase both $n$ and $p$.

When $p$ tends to infinity as $n \lra \infty$, the traditional covariance estimator is not optimal \cite{johnstone2009consistency}.
The sparsity or bandable assumptions on large matrices have been used frequently in the literature.  
Many researchers have studied the large sample properties under the restrictive matrix classes.
\cite{bickel08b} considered the bandable covariance/precision classes and studied the convergence rate of banding estimator on those classes.
\cite{verzelen2010adaptive} derived the convergence rate for precision matrices via sparse Cholesky factors and showed that it is the minimax rate under the Frobenius norm.
In addition, the minimax convergence rates for the sparse or bandable covariance matrices were established by \cite{cai2010optimal}, \cite{cai2012minimax, cai2012optimal} and \cite{xue2013minimax}.
For a comprehensive review on the convergence rate for the covariance and precision matrices, see \cite{cai2016review}.

The posterior convergence rate has been investigated by \cite{pati14}, \cite{banerjee2014}, and \cite{gao2015rate}.
\cite{pati14} showed that their continuous shrinkage priors are optimal for the sparse covariance estimation under the spectral norm in the sense that the posterior convergence rate is quite close to the frequentist minimax rate. They achieved a nearly minimax rate upto a $\sqrt{\log n}$ term under the spectral norm and sparse assumption even when $n = o(p)$. 
\cite{banerjee2014} considered Bayesian banded precision matrix estimation using graphical models. They obtained the posterior convergence rate of the precision matrix under matrix $\ell_\infty$ norm when $\log p = o(n)$.  
\cite{gao2015rate} developed a prior distribution for the sparse PCA and showed that it achieves the minimax rate under the Frobenius norm. They also derived the posterior convergence rate under the spectral norm.

Most of the previous works on the Bayesian estimation of large covariance matrix concentrate on the constrained covariance or precision matrix. To the best of our knowledge, only \cite{gao2016bernstein} considered asymptotic results for large unconstrained covariance matrix under the ``large p and large n" setting. However, they attained the Bernstein-von Mises theorems under somewhat restrictive assumptions on the dimension $p$.


In this paper, we fill the gap in the literature. 
At first, we propose a new decision theoretic framework to define Bayesian minimax rate. 
The posterior convergence rate is the primary concept when the asymptotic optimality is studied in the Bayesian sense. But it is not completely satisfactory.
The following is a quote from \cite{ghosal2017fundamentals} which they write just after defining the posterior convergence rate. 
\begin{quote}\normalsize
	`We defined ``a'' rather than {\it the} rate of contraction, and hence logically any rate slower than a contraction rate is also a contraction rate. Naturally we are interested in a fastest decreasing sequence $\epsilon_n$, but in general this may not exist or may be hard to establish. Thus our rate is an upper bound for a targeted rate, and generally we are happy if our rate is equal to or close to an ``optimal'' rate. With an abuse of terminology we often make statements like ``$\epsilon_n$ is {\it the} rate of contraction.'' ' 
\end{quote}
In the proposed new decision theoretic framework, a probability measure on the parameter space is an action and a prior is a decision rule for it gives a probability measure (the posterior) for a given data set. 
In this setup, we define the convergence rate and the Bayesian minimax rate. 

We investigate the Bayesian minimax rates for unconstrained large covariance matrix. We consider four losses for the covariance inference: spectral norm, Frobenius norm, Bregman divergence and squared log-determinant loss. 
For the spectral norm, we have the complete result of the Bayesian minimax rate. We show that the Bayesian minimax rate is $\min(p/n, 1)$ for all rates of $p$. 
For the Frobenius norm and Bregman divergence, we show the Bayesian minimax lower bound is $p \cdot \min(p, \sqrt{n})/n$ for all rates of $p$, but obtained the upper bound under the constraint $p \leq \sqrt{n}$. Thus, under the condition $p \leq \sqrt{n}$, the Bayesian minimax rate is $p^2/n.$ We also show that the Bayesian minimax rate under the squared log-determinant loss is $p/n$ when $p = o(n)$.

The rest of the paper is organized as follows. In section \ref{prel}, we define the model,  the covariance classes we consider, and introduce some notations. We propose the new  decision theoretical framework and define the Bayesian minimax rate. The Bayesian minimax rates under the spectral norm, the Frobenious norm, the Bregman matrix divergence, and the squared log-determinant loss are presented  in section \ref{main}. A simulation study is given in section \ref{simul}. The discussion is given in section \ref{disc}, and the proofs are given in Supplementary Material (\cite{lee2017supp}).

\section{Preliminaries}\label{prel}
\subsection{The Model and the Inverse-Wishart Prior}
Suppose we observe a random sample from the $p$-dimensional normal distribution
\bean\label{model}
X_1, \cdots, X_n \mid  \sg_n &\overset{iid}{\sim}& N_p (0, \sg_n), 
\eean
where $\sg_n$ is a $p\times p$ positive definite matrix, and $p$ is a function of $n$ such that $p=p_n \lra \infty$ as $n \lra \infty$. The true value of the covariance matrix is denoted by $\sg_0$ or $\sg_{0n}$, which is dependent on $n$. 

For the prior of the covariance matrix $\sg_n$ in model \eqref{model}, we  consider the inverse-Wishart prior 
\bean\label{prior}
\sg_n &\sim& IW_p(\nu_n , A_n),
\eean
where $\nu_n > p - 1$, $A_n$ is a $p \times p$ positive definite matrix for a proper prior. The mean of $\sg_n$ is $A_n/(\nu_n -p -1)$.
The condition $\nu_n >p-1$ is needed for the distribution to have a density in the space of  $p\times p$ positive definite matrices. If $\nu_n$ is an integer with $\nu_n \le p-1$, \eqref{prior} defines a singular distribution on the space of $p\times p$ positive semidefinite matrices \cite{uhlig1994}.

We also consider the truncated inverse-Wishart prior. The inverse-Wishart prior with parameter $\nu$ and $A$ whose  eigenvalues are restricted in $[K_1,K_2]$ with $0 < K_1 < K_2$ is denoted by $IW_p(\nu, A, K_1, K_2)$. The truncated inverse-Wishart prior was adopted for technical reason. By Lemma \ref{BequiF}, to connect the Frobenius norm with Bregman matrix divergence, the eigenvalues of argument matrices have to be bounded. The truncated inverse-Wishart prior guarantees that the posterior covariance matrix has bounded eigenvalues.

\subsection{Matrix Norms and Notations}\label{subsecMN}
We define the spectral norm (or matrix $\ell_2$ norm) for matrices by
\bea
\| A \| := \sup_{\|x\|_2 = 1} \| Ax \|_2,
\eea
where $\| \cdot \|_2$ denotes the vector $\ell_2$ norm defined by $\| x \|_2 := ( \sum_{i=1}^p x_i^2 )^{1/2}$, $x = (x_1,\ldots, x_p)^T \in \bbR^p$ and $A$ is $p \times p$ matrix. 
The spectral norm is the same as $\sqrt{\lambda_{max}(A^TA)}$ or $\lambda_{max}(A)$ if $A$ is symmetric,  where $\lambda_{max}(B)$ denotes  the largest eigenvalue of  $B$.

The Frobenius norm is defined by
\bea
\| A \|_F := \left( \sum_{i=1}^p \sum_{j=1}^p a_{ij}^2 \right)^{\frac{1}{2}}, 
\eea
where $A = (a_{ij})$ is a $p \times p$ matrix. It is the same as  $\sqrt{tr(A^TA)}$, where $tr(B)$ denotes  the trace of $B$. The  Frobenius norm is the vector $\ell_2$ norm with $p \times p$ matrices treated as $p^2$-dimensional vectors.

The Bregman divergence \cite{bregman67} is originally defined for vectors, but it can be extended to the real symmetric matrices.  Let $\phi$  be a differentiable and strictly convex function that maps real symmetric $p\times p$ matrices to $\bbR$.  
The Bregman divergence with $\phi$ between  two real symmetric matrices is defined as
\bea
D_\phi (A,B) := \phi (A) - \phi(B) - tr [ (\nabla \phi(B))^T (A-B) ] ,
\eea
where $A$ and $B$ are real symmetric matrices and $\nabla \phi$ is the gradient of $\phi$, i.e., $\nabla \phi(B) = (\partial \phi(B)/\partial B_{i,j} )$.

In this paper, we consider a class of $\phi$ such that $\phi(X) = \sum_{i=1}^p \varphi (\lambda_i)$ where $\varphi$ is a differentiable and strictly convex real-valued function and $\lambda_i$'s are the eigenvalues of $A$. Furthermore, we assume that  $\varphi$ satisfies the following properties for some constant $\tau_1>0$:
\begin{itemize} 
	\item[(i)] $\varphi$ is a twice differentiable and strictly convex function over $\lambda \in (\tau_1,\infty)$; 
	\item[(ii)]  there exist some constants $C>0$ and $r\in \bbR$ such that $|\varphi(\lambda)| \le C \lambda^r$ for all $\lambda \in (\tau_1,\infty)$; and
	\item[(iii)] for any positive constants $\tau >\tau_1$, there exist some positive constants $M_L$ and $M_U$ such that $M_L \le \varphi''(\lambda) \le M_U$ for all $\lambda \in [\tau_1, \tau]$. 
\end{itemize}
The above class of Bregman matrix divergences includes the squared Frobenius norm, von Neumann divergence and Stein's loss. For their use in statistics and mathematics, see \cite{cai2012optimal}, \cite{dhillon07} and \cite{kulis09}.

If $\varphi(\lambda) = \lambda^2$, the Bregman divergence is the squared Frobenius norm $D_\phi (A,B) = \| A-B \|_F^2$. 
If $\varphi(\lambda) = \lambda \log \lambda - \lambda$,  it is the  von Neumann divergence $ D_\phi(A,B) = tr\left( A\log A - A \log B - A + B \right), $ where $\log A$ is the matrix logarithm, i.e., $A=VDV^T$ is mapped to $\log A = V \log D V^T$. Here, $D=diag(d_i)$ is a $p\times p$ diagonal matrix where $d_i$ is the $i$th eigenvalue of $A$, and $V= [V_1, \cdots, V_p]$ is a $p\times p$ orthogonal matrix where $V_i$ is an eigenvector of $A$ corresponding to the eigenvalue $d_i$. 
If $\varphi(\lambda) = -\log \lambda$, the Bregman divergence is the Stein's loss $D_\phi (A,B) = tr(A B^{-1}) - \log \det(AB^{-1}) -p.$ The Stein's loss is the Kullback-Leibler divergence between two multivariate normal distributions with means zero and covariance matrices $A$ and $B$, respectively.

Finally, we introduce some notations for asymptotic analysis which will be used subsequently. 
For any positive sequences $a_n$ and $b_n$, we say $a_n \asymp b_n$ if there exist positive constants $c$ and $C$ such that $c \le a_n/b_n \le C$ for all sufficiently large $n$. We define
$a_n = o(b_n)$, if $a_n/b_n \ra 0$ as $n\ra \infty$ and 
$a_n = O(b_n)$, if there exist positive constants $N$ and $M$ such that $|a_n| \le M |b_n|$ for all $n \ge N$.
For any random variables $X_n$ and $X$, $X_n \overset{d}{\lra} X$ means the convergence in distribution.
For any real symmetric matrix $A$, $A>0$ ($A \geq 0$) means that the matrix $A$ is positive definite (nonnegative definite). 
We denote $\delta_{A}$ as the dirac measure at $A$.

\subsection{A Class of Covariance Matrices}
Let $\calC_p$ denote the set of all $p\times p$ covariance matrices.
For any positive constants $\tau, \tau_1$ and $\tau_2$, define the class of covariance matrix
\bea
\mathcal{C}(\tau) &=& \cC_p(\tau)\,\,:=\,\, \{ \sg\in \cC_p : \| \sg \| \le \tau , \sg \ge 0 \},\\
\cC(\tau_1,\tau_2) &=& \cC_p(\tau_1,\tau_2) \,\,:=\,\, \{ \sg\in \cC_p : \lambda_{min}(\sg)\ge \tau_1, \|\sg\|\le \tau_2 \},
\eea
where $\lambda_{min}(\sg)$ is the smallest eigenvalue of $\sg$.
Throughout the paper, we consider the model \eqref{model} and assume that the true covariance matrix belongs to $\cC(\tau)$ or $\mathcal{C}(\tau_1, \tau_2)$. 

Often the subgaussian property is used to relax the Gaussian distribution assumption. The distribution of random vector $X$ has subgaussian property with variance factor $\tau > 0$, if 
$$ P ( | v^T ( X - \E X ) | > t ) \le e^{-t^2/(2\tau)} $$
for all $t>0$ and $\| v \|=1$. 
The subgaussian property with variance factor $\tau$ implies $\| \V(X) \|\le 2\tau$.
In the literature, the subgaussian distribution is frequently used as a basic assumption, for examples, \cite{cai2010optimal}, \cite{cai2012minimax, cai2012optimal} and \cite{xue2013minimax}.  
If $X$ follows a multivariate normal distribution, $\| \sg \|\le \tau$ is a sufficient condition for $X$ to have the subgaussian property.

\subsection{Decision Theoretic Prior Selection}\label{subsec:decision}

Let $d(\sg, \sg')$ be a pseudo-metric that measures the discrepancy between two covariance matrices $\sg$ and $\sg'$. 
A sequence $\epsilon_n \lra 0$ is called a posterior convergence rate at the true parameter $\sg_0$ if for any $M_n \lra \infty$,
\bea
\pi( d(\sg, \sg_0) \ge M_n \epsilon_n \mid \bfX_n ) &\lra& 0 
\eea
in $\bbP_{\sg_0}$-probability as $n\lra \infty$.
The convergence rate is measured by the rate of $\epsilon_n$, which allows that the posterior contraction probability converges to zero in probability $\bbP_{\sg_0}$, where $\bbP_{\sg_0}$ is the distribution for random sample $(X_1,\ldots,X_n) \overset{iid}{\sim} N_p(0, \sg_0)$.
In the literature, the posterior is said to achieve the minimax rate if its convergence rate is the same as the frequentist minimax rate (\cite{pati14}; \cite{gao2015rate}; \cite{hoffmann2015adaptive}).
Since the posterior convergence rate cannot be faster than the frequentist minimax rate (\cite{hjort2010bayesian}), it is often called the optimal rate of posterior convergence (\cite{shen2015adaptive}; \cite{rockova2015bayesian}). 
However, its definition is  elusive as the quote from \cite{ghosal2017fundamentals} indicates.

As an alternative framework for the evaluation of the prior and the posterior, we take a frequentist decision theoretical approach. For each $n$, the parameter space is $\calC_p$ and the action space is the set of all probability measures on $\calC_p$. After the data $\bfX_n$ is collected, the posterior $\pi(\cdot | \bfX_n)$ is computed for the given prior $\pi$ and the posterior takes a value in the action space. 
In this setup, the prior can be considered as a decision rule, because the prior and observations together produce the posterior. 
A probability measure in the action space will be used as a posterior for the inference, but it does not have to be generated from a prior. 
We define the loss and risk function of the parameter $\Sigma_0$ and the prior $\pi$ as 
\bea
\calL(\Sigma_0, \pi(\cdot | \bfX_n)) &:=& \E^\pi \big( d(\Sigma, \Sigma_0) | \bfX_n), \\
\calR(\Sigma_0, \pi) &:=& \E _{\Sigma_0} \calL(\Sigma_0, \pi(\cdot | \bfX_n)) \,\,=\,\, \E _{\Sigma_0} \E^\pi \big( d(\Sigma, \Sigma_0) | \bfX_n).
\eea
Note that the risk function measures the performance of the prior $\pi$.
To distinguish them from the usual loss and risk, we call the above loss and risk as {\it posterior loss (P-loss)} and {\it posterior risk (P-risk)}. 
The P-risk itself is not new. For example, the P-risk was also used in \cite{castillo2014bayesian} for density estimation on the unit interval.


There are a couple of benefits of the proposed decision theoretic prior selection. 
First, the decision theoretic prior selection makes the definition of the minimax rate of the posterior mathematically concrete.  
Although the minimax rate of the posterior is used frequently, it has been used without a rigorous definition. The frequentist minimax rate is used as a proxy of the desired concept.
Second, in the study of the posterior convergence rate, the scale of the loss function needs to be carefully chosen so that the posterior consistency holds. But in the proposed decision theoretic prior selection, the inconsistent priors can be compared without any conceptual difficulty. 
Thus, the scale of the loss function does not need to be chosen.

We now define the  minimax rate and convergence rate for P-loss.
Let $\Pi_n$ be the class of all priors on $\sg_n$.
A sequence $r_n$ is said to be the {\it minimax rate for P-loss (P-loss minimax rate)} or simply {\it the Bayesian minimax rate} for the class $\cC_p^* \subset \cC_p$ and the space of the prior distributions $\Pi_n^* \subset \Pi_n$, if 
$$ \inf_{\pi\in\Pi_n^*} \sup_{\sg_0\in\cC_p^*}  \E_{\Sigma_0} \calL(\Sigma_0, \pi(\cdot | \bfX_n)) \asymp r_n.$$
A prior $\pi^*$ is said to have a {\it convergence rate for P-loss (P-loss convergence rate)} or {\it convergence rate}  $a_n$, if 
$$ \sup_{\sg_0\in\cC_p}  \E_{\Sigma_0} \calL(\Sigma_0, \pi^*(\cdot | \bfX_n)) \lesssim a_n,$$
and, if $a_n \asymp r_n$ where $r_n$ is the minimax rate for P-loss, $\pi^*$ is said to attain the minimax rate for P-loss or the Bayesian minimax rate.  If it is clear from context, we will drop P-loss and refer them as the minimax rate and the convergence rate. 
For a given inference problem, we wish to find a prior $\pi^*$ which attains the minimax rate for P-loss.

\begin{remark}
	The P-loss convergence rate implies the posterior convergence rate by Proposition \ref{epost covrate} in Supplementary Material (\cite{lee2017supp}). By obtaining the P-loss convergence rate, we also get the traditional posterior convergence rate. The converse may not be true, because for certain loss functions, the P-loss may not even converge to $0$ while the posterior convergence rate converges to $0$.
\end{remark}

\begin{remark}
	The P-loss convergence rate is slower than or equal to the frequentist minimax rate by Proposition \ref{freq bayes LB} in Supplementary Material (\cite{lee2017supp}). To obtain a P-loss minimax lower bound, the mathematical tools for frequentist minimax lower bound can be used. 
\end{remark}

\begin{remark}
If we assume that the prior class $\Pi_n$ includes the data dependent priors, the P-loss minimax rate is the same as the frequentist minimax rate. Take $\pi = \delta_{\hat{\sg}^*}$ where $\hat{\sg}^*$ is an estimator attaining the frequentist minimax rate. Then, $\pi$ attains the frequentist minimax rate and thus attains the Bayesian minimax rate. 
However, the data-dependent prior is not acceptable for legitimate Bayesian analysis unless the prior is dependent on ancillary statistics. 
Even if $\Pi_n$ does not contain data-dependent priors, in most cases the frequentist and P-loss minimax rates are the same.

However, if we consider a restricted class of priors, the P-loss minimax rate might differ from the usual frequentist minimax rate. In such cases, the frequentist minimax rate will not be a natural concept to study the asymptotic properties of the posterior. 
See Remark in subsection \ref{subsec:fro}.
\end{remark}

\section{Bayesian Minimax Rates under Various Matrix Loss Functions}\label{main}
\subsection{Bayesian Minimax Rate under Spectral Norm}
In this subsection, we show that the Bayesian minimax rate for covariance matrix under the spectral norm is $\min(p/n, 1)$. We also show that the prior
\bean\label{mixWandI}
\pi_n(\sg_n) &=& IW_p(\sg_n \mid \nu_n, A_n) I\left(p \le \frac{n}{2}\right) + \delta_{I_p}(\sg_n) I\left(p > \frac{n}{2}\right)
\eean
attains the Bayesian minimax rate for the class $\cC(\tau_1, \tau_2)$ under the spectral norm, where $IW_p(\sg \mid \nu_n, A_n)$ is the inverse-Wishart distribution, $\nu_n>p-1$ and $A_n$ is a $p\times p$ positive definite matrix. 
We have the complete result for all values of $n$ and $p$.
The Bayesian minimax rate holds for any $n$ and $p$, regardless of their relationship. 
The number $1/2$ in the prior  \eqref{mixWandI}  can be replaced by any number in $(0,1)$ and the prior still renders the minimax rate.

The main result of the section is given in Theorem \ref{minimaxS} whose proof is given in Supplementary Material (\cite{lee2017supp}). 
We divide the proof into two parts: lower bound  and  upper bound parts.
First, we show that the lower bound of the frequentist minimax rate is $\min(p/n,1)$, which may be of interest in its own right,  and it in turn implies that $\min(p/n,1)$ is a Bayesian minimax lower bound.  
After that, the P-loss convergence rate with the prior \eqref{mixWandI} is derived, which is the same as the Bayesian minimax lower bound when $\nu_n^2=O(np)$ and $A_n = S_n$. 
Consequently, we obtain the following theorem by combining these two results.
Throughout the paper, $\Pi_n$ is the class of all priors on $\sg_n \in \calC_p$ as we have defined in subsection \ref{subsec:decision}.

\begin{theorem}\label{minimaxS}
	Consider the model \eqref{model}. For any positive constants $\tau_1<\tau_2$, 
	$$\inf_{\pi\in\Pi_n} \sup_{\sg_0 \in \cC(\tau_1 , \tau_2)} \Esg \bbE^{\pi} ( \|\sg_n - \sg_0 \|^2 \mid \bfX_n) \asymp  \min\left(\frac{p}{n}, 1 \right).$$
	Furthermore, the prior \eqref{mixWandI} with $\nu_n^2 = O(np)$ and $\|A_n\|^2 = O(np)$ attains the Bayesian minimax rate.
\end{theorem}

\begin{remark}
	The proof for the lower bound holds even for $\tau_1$ and $\tau_2$ depending on $n$ and possibly for $\tau_1 \lra 0$ and $\tau_2\lra\infty$ as $n\lra \infty$. In such cases, 
	 the rate of the minimax lower bound is $\tau_{2}^2\cdot \min\left(p/n, 1 \right)$. 
	For details, see Theorem \ref{LboundS} in the Supplementary Material (\cite{lee2017supp}).
	Note that $\tau_2$ affects the minimax lower bound, while $\tau_1$ does not.
	A similar phenomenon occurs for estimation of sparse spiked covariance matrices. See Theorem 4 of \cite{cai2016review}.
\end{remark}

We have complete results of the Bayesian minimax rate under the spectral norm. In words, the results above do not have any condition on the rate of $p$ and $n$. For a given rate of $p$, we obtained the Bayesian minimax rate. 
When $p$ grows the same rate as $n$, the above theorem shows that estimating the covariance under the spectral norm is hopeless. Indeed, this can be seen from the form of the prior \eqref{mixWandI}. When $p \geq n/2$, the point mass prior $\delta_{I_p}$ gives the Bayesian minimax rate. In words, you can not do better than the useless point mass prior $\delta_{I_p}$. 

Applying techniques used in the proof of the upper bound, one can show that the prior \eqref{mixWandI} also gives the same P-loss convergence rate for precision matrix. 
\begin{corollary}\label{UboundS_Pre}
	Consider the model \eqref{model} and prior \eqref{mixWandI} with $\nu_n^2 = O(np)$ and $\|A_n\|^2 = O(np)$. For any positive constants $\tau_1<\tau_2$,
	$$\sup_{\sg_0 \in \cC(\tau_1 , \tau_2)} \Esg \bbE^{\pi} ( \|\sg_n^{-1} - \sg_0^{-1} \|^2 \mid \bfX_n) \le c \cdot \min\left(\frac{p}{n}, 1 \right)$$
	for all sufficiently large $n$ and some constant $c>0$.
\end{corollary}

We remark here that \cite{gao2016bernstein} derived a posterior convergence rate for unconstrained covariance matrix under the spectral norm when $p=o(n)$.
In this paper, we obtained a P-loss convergence rate which implies the stronger convergence than a posterior convergence rate, for any $n$ and $p$. \cite{gao2016bernstein} also attained a posterior convergence rate for precision matrix under $p^2=o(n)$. In this paper, Corollary \ref{UboundS_Pre} gives a P-loss convergence rate for any $n$ and $p$.

\subsection{Bayesian Minimax Rate under Frobenius Norm}\label{subsec:fro}

Throughout this subsection, $\tau>0$ can depend on $n$ and possibly $\tau\lra\infty$ as $n\lra\infty$. 
In this subsection, we show that the rate of the Bayesian minimax lower bound for covariance matrix under Frobenius norm is $\tau^2 \cdot \min(p,\sqrt{n})\cdot p/n$ for the class $\cC(\tau)$, and the inverse-Wishart prior attains the Bayesian minimax lower bound when $p\le \sqrt{n}$.

The following theorem gives the Bayesian minimax lower bound. 
The proof of Theorem \ref{LboundF} is given in Supplementary Material (\cite{lee2017supp}).
In the proof of the theorem, we prove that the lower bound of the  frequentist minimax rate is $ \tau^2\cdot \min(p,\sqrt{n})\cdot p/n$ as a by-product. 
\begin{theorem} \label{LboundF}
	Consider the model \eqref{model}. For any $\tau >0$,  
	$$\inf_{\pi\in\Pi_n} \sup_{\sg_0 \in \mathcal{C}(\tau)} \Esg \bbE^{\pi} (\| \sg_n - \sg_0 \|_F^2 \mid \bfX_n) \ge c \cdot \tau^2 \cdot \frac{p}{n}\cdot \min(p, \sqrt{n})$$
	for all sufficiently large $n$ and some constant $c > 0$.
\end{theorem}

\begin{theorem}\label{UboundW}
	Consider the model \eqref{model} and prior \eqref{prior} with 
	$ \nu_n > 0$ and $A_n >0$ for all $n$. If $\nu_n = p$ and $\|A_n\|^2 =O(n )$, for any $\tau >0$, 
	$$\sup_{\sg_0 \in \mathcal{C}(\tau)}\bbE_{\sg_0}\bbE^{\pi} (\|\sg_n - \sg_0 \|^2_F \mid \bfX_n) \le c \cdot \tau^2 \cdot \frac{p^2}{n} $$
	for some constant $c>0$ and all sufficiently large $n$. Furthermore, if $p\le \sqrt{n}$, $\nu_n^2= O(np)$ and $\|A_n\|^2 = O(np)$ is the necessary and sufficient condition for achieving the rate $p^2/n$. 
\end{theorem}

Note that if $\tau>0$ is a fixed constant, from the relationship between the spectral norm and Frobenius norm, one can obtain a P-loss convergence rate $\min(p,n)\cdot p/n$ instead of $p^2/n$ in Theorem \ref{UboundW}. However, in this case, one should restrict the parameter space to $\calC(\tau_1,\tau_2)$ instead of the more general parameter space $\calC(\tau)$. 

In practice, we recommend using $\nu_n=p$ and small $A_n$ such as $A_n=O_p$ or $A_n=I_p$, where $O_p$ denotes a $p\times p$ zero matrix because it guarantees the rate $p^2/n$ regardless of the relation between $n$ and $p$.  
Note that the Jeffreys prior \cite{Jeffreys61theory}
\bea
\pi(\sg_n) &\propto& det (\sg_n)^{-(p+2)/2},
\eea
the independence-Jeffreys prior \cite{sun2007objective}
\bea
\pi(\sg_n) &\propto& det (\sg_n)^{-(p+1)/2}
\eea
and the prior proposed by \cite{geisser1963posterior}
\bea
\pi(\sg_n) &\propto& det (\sg_n)^{-p}
\eea
satisfy the above conditions. They can be viewed as inverse-Wishart priors, $IW(\nu_n, A_n)$, with parameters $(1,O_p)$, $(0, O_p)$ and $(p-1, O_p)$, respectively. Furthermore, the $IW(p+1, S_n)$ prior, whose mean is $S_n$, also satisfies the conditions in Theorem \ref{UboundW}.

By Theorem \ref{UboundW} and Theorem \ref{LboundF}, we have the Bayesian minimax rate $\tau^2\cdot p^2/n$ for covariance matrix under the Frobenius norm when $p\le \sqrt{n}$.
Thus, with the inverse-Wishart prior, we attain the Bayesian minimax rate under the Frobenius norm.  
\begin{theorem} \label{minimaxF}
	Consider the model \eqref{model}. If $p\le \sqrt{n}$, for any $\tau >0$,
	$$\inf_{\pi\in\Pi_n} \sup_{\sg_0 \in \mathcal{C}(\tau)} \Esg \bbE^{\pi} (\| \sg_n - \sg_0 \|_F^2 \mid \bfX_n) \asymp \tau^2 \cdot \frac{p^2}{n}.$$
	Furthermore,  $\nu_n^2= O(np)$ and $\|A_n\|^2 = O(np)$ is the necessary and sufficient condition for the prior \eqref{prior} to achieve the Bayesian minimax rate when $p\le \sqrt{n}$. 
\end{theorem}

\begin{remark} 
In section 2.4, we have said that the Bayesian minimax rate can be different from the frequentist minimax rate when a restricted prior class is considered, and that the frequentist minimax rate will not be a natural concept to address the asymptotic properties of the posteriors from a restricted prior class. We give an example here.  Consider a prior class $\Pi_n^* = \{ \pi \in IW_p(\nu_n, A_n) : \nu_n \ge n , \, A_n \in\calC_p  \}$ and assume $p\le \sqrt{n}$. It is easy to check that 
	\bea
	\inf_{\pi\in\Pi_n^*} \sup_{\sg_0 \in \mathcal{C}(\tau)} \bbE_{\sg_0} \bbE^\pi (\| \sg_n - \sg_0 \|_F^2 \mid \bfX_n) 
	&\asymp& \tau^2 \cdot p,
	\eea
	from the proof of Theorem \ref{UboundW}.
Note that the obtained P-loss minimax rate differs from the usual frequentist minimax rate, $\tau^2 \cdot p^2/n$. 
\end{remark}

\subsection{Bayesian Minimax Rate under Bregman matrix Divergence}
In this section, we obtain the Bayesian minimax rate under a certain class of Bregman matrix divergences.
Let $\Phi$ be the class of differentiable and strictly convex real-valued functions satisfying (i)-(iii) conditions in the subsection \ref{subsecMN}, and let $\calD_\Phi$ be the class of Bregman matrix divergences $D_\phi$ where $\phi(X) = \sum_{i=1}^p \varphi (\lambda_i)$ for symmetric matrix $X$ and $\varphi \in \Phi$. 

To achieve the Bayesian minimax convergence rate for Bregman matrix divergences, we use the truncated inverse-Wishart distribution $IW_p( \nu_n,A_n,K_1, K_2)$ whose eigenvalues are all in $[K_1,K_2]$ for some positive constants $K_1<K_2$. 
The density function of $IW_p(\nu_n,A_n,K_1,K_2)$ is given by
\bean\label{tr_IWK12}
\pi^{n,K_1,K_2}(\sg_{n}) &=& \frac{ \det(\sg_{n})^{-(\nu+p+1)/2} e^{-\frac{1}{2}tr(A_n\sg_{n}^{-1})}I(\sg_{n}\in \cC(K_1,K_2) ) }{ \int_{\cC(K_1,K_2)} \det(\sg_{n}')^{-(\nu+p+1)/2} e^{-\frac{1}{2}tr(A_n\sg_{n}'^{-1})} d\sg_n' }
\eean
where $\nu_n>p-1$ and $A_n$ is a $p\times p$ positive definite matrix.

\begin{theorem}\label{minimaxB}
	Consider the model \eqref{model}. If $p\le \sqrt{n}$, for any positive constants $\tau_1<\tau_2$
	$$\inf_{\pi\in\Pi_n} \sup_{\sg_0 \in \cC(\tau_1 , \tau_2)} \Esg \bbE^{\pi} ( D_\phi (\sg_n, \sg_0) \mid \bfX_n) \asymp \frac{p^2}{n}$$
	for all $D_\phi \in \calD_\Phi$. Furthermore, the prior \eqref{tr_IWK12} with $\nu_n^2= O(np)$, $\|A_n\|^2 = O(np)$, $K_1 < \tau_1$ and $K_2 > \tau_2$ achieves the Bayesian minimax rate when $p\le \sqrt{n}$. 
\end{theorem}

To extend the minimax result for the squared Frobenius norm to the Bregman matrix divergence, the posterior distribution for $\sg_n$ and the true covariance $\sg_0$ should be included in the class $\cC(K_1,K_2)$ and $\cC(\tau_1,\tau_2)$, respectively, for some positive constants $K_1 <\tau_1$ and $K_2 >\tau_2$.
The truncated inverse-Wishart prior was needed to restrict the posterior distribution for $\sg_n$ within the class $\cC(K_1,K_2)$. 
In practice, we recommend using sufficiently small $K_1$ and large $K_2$. 
According to the above theorem, the minimax convergence rate for the class $\calD_\Phi$ is equivalent to that for the Frobenius norm if we consider the parameter space $\cC(\tau_1 , \tau_2)$. Moreover, the truncated inverse-Wishart prior $IW_p(\nu_n,A_n,K_1, K_2)$ achieves the Bayesian minimax rate. The proof of the theorem is given in Supplementary Material (\cite{lee2017supp}).

\subsection{Bayesian Minimax Rate of Log Determinant of Covariance Matrix}
In this subsection, we establish the Bayesian minimax rate for the log-determinant of the covariance matrix under squared error loss. The frequentist minimax lower bound was derived by \cite{cai2015law}. We prove that the inverse-Wishart prior achieves the Bayesian minimax rate when $p=o(n)$.

The estimator of the log-determinant of the covariance matrix can be used as a basic ingredient  for  constructing hypothesis test or the quadratic discriminant analysis \cite{anderson03}.
The log-determinant of the covariance matrix is needed to compute the quadratic discriminant function for multivariate normal distribution
$$-\frac{1}{2}\log \det \sg - \frac{1}{2}(x-\mu)^T \sg^{-1} (x -\mu) $$
where $x$ is the random sample from $N_p (\mu, \sg)$.
Furthermore, the differential entropy of $N_p (\mu, \sg)$ is given by
$$\frac{p}{2} + \frac{p \log (2\pi)}{2} + \frac{ \log\det \sg }{2},$$
so the estimation of the differential entropy is equivalent to estimation of the log-determinant of the covariance matrix, when we consider the multivariate normal distribution. 
The differential entropy has various applications including independent component analysis (ICA), 
spectroscopy, image analysis,  
and information theory.
See \cite{beirlant97}, \cite{dudewicz91}, \cite{Hyvarinen98} and \cite{Cover91}.

\cite{cai2015law} showed that the minimax rate for the log-determinant of the covariance matrix under squared error loss is $p/n$ and their estimator achieves this optimal rate when $p = o(n)$. 

On the Bayesian side, \cite{srivastava2008bayesian} and \cite{gupta2010parametric} suggested a Bayes estimator for log-determinant of the covariance matrix of the multivariate normal. They proposed using the inverse-Wishart prior and showed that the posterior mean minimizes expected Bregman divergence. 
In this subsection, we support their argument by showing that the inverse-Wishart prior achieves the P-loss minimax rate for log-determinant of the covariance matrix under squared error loss. Thus, we show that the inverse-Wishart prior gives the optimal result in the Bayesian sense. We also show the sufficient conditions for achieving the Bayesian minimax rate. 
The following theorem presents the Bayesian minimax rate for the log-determinant of the covariance matrix under the squared error loss. The proof of the theorem is given in Supplementary Material (\cite{lee2017supp}).

\begin{theorem}\label{minimaxLD}
	Consider the model \eqref{model}. If $p = o(n)$, we have
	$$\inf_{\pi\in\Pi_n} \sup_{\sg_0 \in \cC_p} \Esg \bbE^{\pi} ( (\log \det \sg_n - \log \det \sg_0)^2 \mid \bfX_n) \asymp \frac{p}{n}.$$
	Furthermore, prior \eqref{prior} with $ \nu_n^2 = O( n/p )$ and $A_n =O_p$ attains the Bayesian minimax rate.
\end{theorem}

\begin{remark}
	One can also show that the optimal minimax convergence rate is achieved by using the prior \eqref{prior} with $\nu_n^2= O(n/p)$, $A_n=c_n S_n$ and $c_n^2 =O(n/p)$.
\end{remark}

\begin{remark}
\cite{gao2016bernstein} showed the Bernstein-von Mises result for the log-determinant of covariance, which implies a posterior convergence rate. However, they considered a restrictive parameter space $\calC(\tau_1,\tau_2)$ and the stronger condition $p^3 = o(n)$. In this paper, the more general parameter space $\calC_p$ and weaker condition $p=o(n)$ are sufficient for the stronger result, a P-loss convergence rate.
\end{remark}

\section{Simulation study}\label{simul}
In this section, we support our theoretical results by a simulation study. The simulations for three loss functions, spectral norm, square of scaled Frobenius norm and squared log-determinant loss, were conducted. 
We compare the performance of the minimax priors with those of some frequentist estimators. 

We choose the posterior mean as a Bayesian estimator. The posterior mean obtained from the minimax prior  attains the minimax rate in Theorem \ref{UboundS}, Theorem \ref{UboundW} and Theorem \ref{minimaxLD} by the Jensen's inequality.

We generated dataset $X_1,\ldots, X_n$ from $N_p(0, \sg_0)$ where true covariance matrix $\sg_0$ was either diagonal or full covariance matrix. A full covariance matrix is a covariance matrix which does not have any restriction on its elements such as sparsity or banding. 
In the diagonal covariance setting, the true covariance is $\sg_0 = diag(\sigma_{0, ii})$ where $\sigma_{0,ii} \overset{iid}{\sim} Unif(0, 5)$. In the full covariance setting, we made the true covariance $\sg_0 = V^T V$ where $V=(v_{ij})$ is a $p\times p$ matrix with $v_{ij} \overset{iid}{\sim} N(0, 5/p)$.  
In the simulation study, the dimensions of the true covariance matrices are $25, 50, 100$ and $200$, and the numbers of data $n$ are either $n=p^2$ or $n=\lceil p^{3/2} \rceil$.
For each setting, we generated a true covariance once for which we generated 100 data sets and calculated estimators of the covariance. 

For the spectral norm and square of scaled Frobenius norm loss, we computed the posterior mean of the inverse-Wishart prior, $IW(\nu_n, A_n)$,  for comparison. We chose $\nu_n = 2, \sqrt{n/p}, p$ and $n$ to see the effect of the $\nu_n$, but fixed $A_n = O_p$ to remove the prior effect on the structure of the covariance estimate. 
By Theorems \ref{UboundS} and \ref{UboundW}, when $n=p^2$,  the inverse-Wishart prior with $\nu_n=2, \sqrt{n/p}$ and $p$ are minimax priors, while that with $\nu_n=n$ is not. 
We also computed the sample covariance $S_n$ and the tapering estimator $\hat{\Sigma}_k$ \cite{cai2010optimal} for comparison. As mentioned before, the sample covariance matrix is a Bayesian estimator using inverse-Wishart prior with $\nu_n =p+1$ and $A_n = O_p$, which satisfies the conditions in Theorem \ref{UboundW}.
We used $k= \sqrt{n}$ as the threshold of tapering estimator. It corresponds to  $\alpha=0$ in \cite{cai2010optimal}, which gives the minimal sparse constraint for the covariance matrix in their class.


\begin{figure}[!bt]\centering
	\includegraphics[width=12cm, height=12cm]{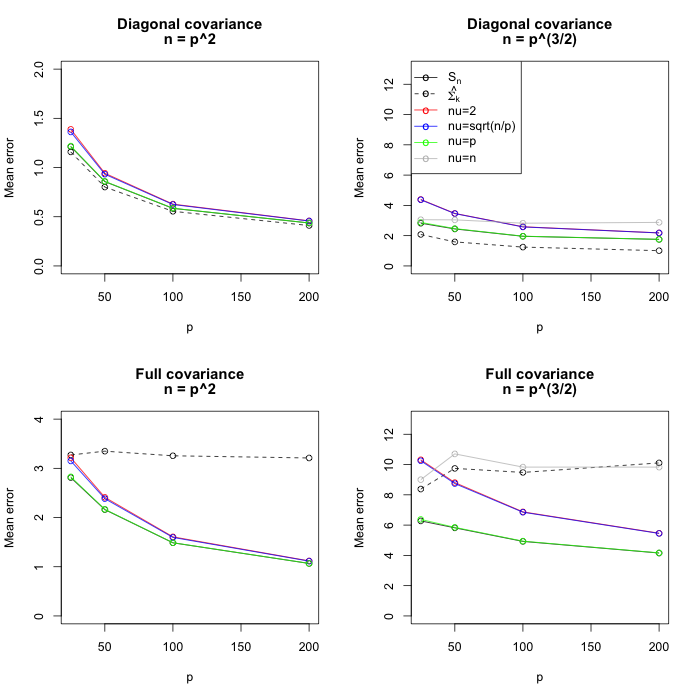}
	\caption{The risks for the Bayes estimator with $IW(\nu_n, O_p, K)$, the sample covariance $S_n$ and tapering estimator $\hat{\Sigma}_k$ under the spectral norm loss function.  
		The true covariances were generated in diagonal setting (top row) and full covariance setting (bottom row). The number of the observation was chosen by either $n=p^2$ (left column) or $n=\lceil p^{3/2}\rceil$ (right column). 
	}
	\label{fig:covSpec}
\end{figure}
Figure \ref{fig:covSpec} summarizes the simulation results for the spectral norm. 
Each point of the plot was calculated by
\bea
\frac{1}{100} \sum_{s=1}^{100} \|\sg_0 - \what{\sg}_n^{(s)} \|
\eea
where $\what{\sg}_n^{(s)}$ is the estimate of the true covariance $\sg_0$ in $s$-th simulation. 
The first and second rows of Figure \ref{fig:covSpec} show  the results when the true covariance matrix is a diagonal and full covariance, respectively; the left and right columns are the results when  $n=p^2$ and $n=\lceil p^{3/2} \rceil$, respectively.

The inverse-Wishart prior with $\nu_n = p$ and the sample covariance performed well in all cases. They are either the best or comparable to the best. When $n= \lceil p^{3/2} \rceil$, the  truncated inverse-Wishart prior with $\nu_n = n$  is not minimax, and the simulation results show that it  performed the worst or the second to the worst. 
The inverse-Wishart priors with $\nu_n = 2$ and $\sqrt{n/p}$ are minimax, and thus their risks decrease as $n \lra \infty$ in all cases, but their performance are slightly worse than that with $\nu_n = p$. 
The tapering estimator $\hat{\sg}_k$ performed the best in diagonal settings because it gives zero to many of upper and lower diagonal elements or shrink them toward zero. However, in the full covariance settings, it performed the worst or close to the worst  for the same reason.


\begin{figure}[!bt]\centering
	\includegraphics[width=12cm, height=12cm]{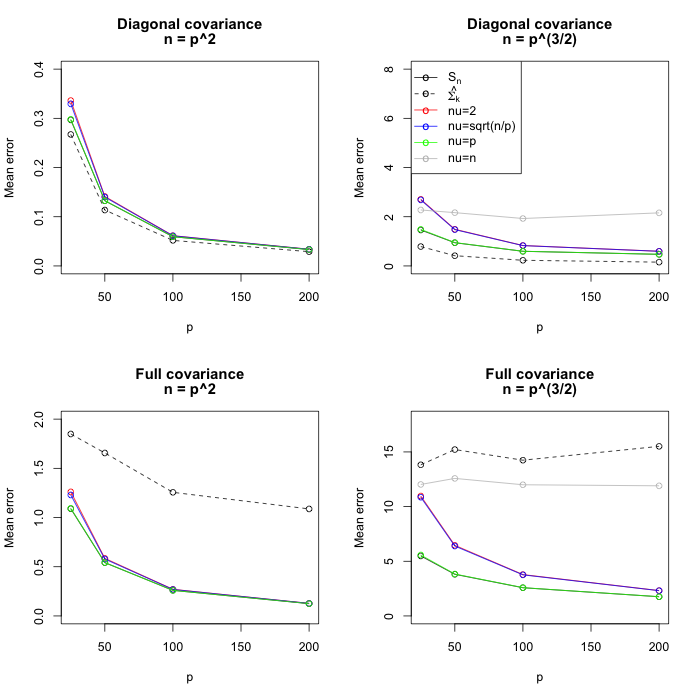}
	\caption{The risks for the Bayes estimator with $IW(\nu_n, O_p, K)$, the sample covariance $S_n$ and tapering estimator $\hat{\Sigma}_k$ under the squared Frobenius norm loss function.
		The true covariances were generated in diagonal setting (top row) and full covariance setting (bottom row). The number of the observation was chosen by either $n=p^2$ (left column) or $n=\lceil p^{3/2}\rceil$ (right column).
	}
	\label{fig:covFro}
\end{figure}
Figure \ref{fig:covFro} summarizes the simulation results for Frobenius norm. 
Each point of the plot was calculated by
$$\frac{1}{100} \sum_{s=1}^{100} \frac{1}{p} \| \sg_0 - \widehat{\sg}_n^{(s)} \|_F^2 $$
where $ \widehat{\sg}_n^{(s)}$ is the estimate of the true covariance $\sg_0$ in $s$-th simulation. 
The results are quite similar to the spectral norm case.


For the square of log-determinant loss, we chose the maximum likelihood estimator (MLE) $\log \det S_n$ and the uniformly minimum variance unbiased estimator (UMVUE) for comparison. The UMVUE of $\log \det \sg$ is given by
$$ \log \det S_n + p\log \left( \frac{n}{2} \right) - \sum_{j=0}^{p-1} \psi \left( \frac{n-k}{2} \right) $$
where $\psi$ is the digamma function which is defined by $\psi(x) = d/dz \log \Gamma(z)|_{z=x}$ where $\Gamma$ is the gamma function. 
See \cite{ahmed89} for more details. 
We tried the same settings for inverse-Wishart prior as before. Note that for $n=p^2$ and $n=\lceil p^{3/2}\rceil$, the choices $\nu_n = 2$ and $\sqrt{n/p}$ satisfy the sufficient condition in Theorem \ref{minimaxLD} while $\nu_n =p$ and $n$ do not. 
The posterior mean of the log-determinant for the inverse-Wishart prior is
$$\log \det \left(S_n+ \frac{A_n}{n}\right) + p\log \left( \frac{n}{2} \right) - \sum_{j=0}^{p-1} \psi \left( \frac{n +\nu_n -k}{2} \right) .$$
Thus, the UMVUE is the same as the Bayesian estimator using inverse-Wishart prior with $\nu_n =0$ and $A_n = O_p$, which satisfies the sufficient condition in Theorem \ref{minimaxLD}.

\begin{figure}[!bt]\centering
	\includegraphics[width=12cm, height=12cm]{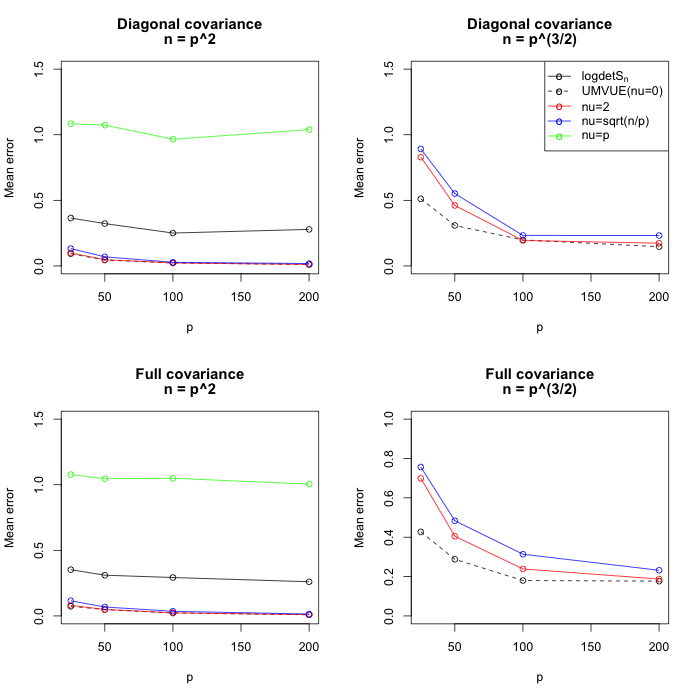}
	\caption{The squared log-determinant loss function plot. The true covariances were generated in diagonal setting (top row) and full covariance setting (bottom row). The number of the observation was chosen by either $n=p^2$ (left column) or $n=\lceil p^{3/2}\rceil$ (right column).}
	\label{fig:covLogdet}
\end{figure}
Figure \ref{fig:covLogdet} summarizes the simulation results for log-determinant.
Each point of the plot was calculated by
$$\frac{1}{100} \sum_{s=1}^{100}  ( \log \det \sg_0 - \widehat{\log \det \sg}_n^{(s)} )^2 $$
where $\widehat{\log \det \sg}_n^{(s)}$ is the estimate of $\log\det\sg$ in $s$-th simulation and $\sg_0$ is the true covariance. 
The top and bottom rows are for the diagonal and full true covariance cases, respectively; the left and right columns are for $n=p^2$ and $\lceil p^{3/2}\rceil$, respectively. 

For the squared log-determinant loss,  the inverse-Wishart priors with $\nu_n=2$ and $\sqrt{n/p}$ are minimax, while those with $\nu_n = p$ and $n$ are not. The UMVUE or the Bayes estimator of the  the inverse-Wishart priors with $\nu_n=0$ performed the best in all cases.  The inverse-Wishart priors with $\nu_n=2$ and $\sqrt{n/p}$ performed comparable to the UMVUE.  Interestingly, the inverse-Wishart priors with $\nu_n=p$, which was the best under the spectral norm,  performed worst in all cases. When $n = \lceil p^{3/2}\rceil$, the results for $\nu_n=p$ do not appear in the Figure \ref{fig:covLogdet} because of its large risk values.
This signifies the fact that we need to choose different prior parameter for different loss function.

\section{Discussion}\label{disc}
In this paper, we develop a new framework for the Bayesian minimax theory, and introduce Bayesian minimax rate and P-loss convergence rate. The proposed decision theoretic framework  gives an alternative way to distinguish the good priors from the inadequate ones and makes the definition of the minimax rate of the posterior clear. 
We obtain the Bayesian minimax rates for the normal covariance model under the various loss functions: spectral norm, the squared Frobenius norm, Bregman matrix divergence and squared log-determinant loss for large covariance estimation. We show that the inverse-Wishart prior or truncated inverse-Wishart prior attains the Bayesian minimax rate. 
The simulation results support the theory obtained.


\appendix
\section{Basic properties of P-loss convergence rate}

A frequentist minimax lower bound is defined as a lower bound of 
$$\inf_{\hat{\sg}} \sup_{\sg_0\in \cC_p } \bbE_{\sg_0} (d(\hat{\sg}, \sg_0) )$$ 
where $\hat{\sg}$ denotes an arbitrary estimator of $\sg_0$, and we say $r_n$ is the frequentist minimax rate for the class $\cC_p$ and the space of the estimators of $\sg_0$, if
\bea
\inf_{\hat{\sg}} \sup_{\sg_0\in \cC_p } \bbE_{\sg_0} (d(\hat{\sg}, \sg_0) ) \asymp r_n .
\eea
Propositions \ref{epost covrate} and \ref{freq bayes LB} state two basic properties of P-loss convergence rate and the Bayesian minimax rate. 
\begin{proposition}\label{epost covrate}
	For any $\sg_0 \in \calC_p$,  a P-loss convergence rate at $\sg_0$ is a posterior convergence rate at $\sg_0$.
\end{proposition}
\begin{proof}
	Suppose that the rate of the P-loss convergence rate at $\sg_0\in \calC_p$ is $\epsilon_n$, i.e.,
	\bea
	\Esg \bbE^{\pi} ( d( \sg , \sg_0 ) | \bfX_n) &\asymp& \epsilon_n.
	\eea
	For a sequence $M_n\lra \infty$ and $\delta>0$,
	\bea
	\bbP_{\sg_0}\left( \pi(d(\sg,\sg_0) \ge M_n \epsilon_n \mid \bfX_n) > \delta \right)
	&\le& \bbP_{\sg_0} \left( \bbE^\pi(d(\sg,\sg_0) \mid \bfX_n) > \delta M_n \epsilon_n  \right) \\
	&\le& \frac{1}{\delta M_n \epsilon_n} \Esg \bbE^\pi (d(\sg,\sg_0) \mid \bfX_n) \\
	&\lra& 0, \quad \text{ as } n\to\infty.
	\eea
	The first and second inequalities follow from the Markov inequality. 
\end{proof}

\begin{proposition}\label{freq bayes LB}
	A frequentist minimax lower bound for $\sg_0$ is also a P-loss minimax lower bound  for any loss function $d(\cdot,\sg_0)$, i.e.,
	\bea
	\inf_{\pi \in \Pi_n}\sup_{\sg_0 \in \calC_p} \bbE_{\sg_0}\bbE^\pi ( d(\sg,\sg_0)\mid \bfX_n ) &\ge& \inf_{\hat{\sg}}\sup_{\sg_0 \in \calC_p} \bbE_{\sg_0} ( d(\hat{\sg},\sg_0) ),
	\eea
	where $\hat{\sg}$ denotes an arbitrary estimator of $\sg_0$.
\end{proposition}
\begin{proof}
	Note that the P-risk is always equal or larger than the posterior convergence rate by Markov's inequality, and the frequentist minimax rate is a lower bound for the posterior convergence rate (\cite{hjort2010bayesian}). Thus, the frequentist minimax rate is also a lower bound for the P-loss minimax rate.
\end{proof}

\section{Proof of Theorem \ref{minimaxS}}

We divide the proof of Theorem \ref{minimaxS} into two parts: the lower bound part (Theorem \ref{LboundS}) and the upper bound part (Theorem \ref{UboundS}).
For Theorem \ref{LboundS}, we have a quite strong result in sense that it holds even for $\tau_1$ and $\tau_2$ depending on $n$ and possibly $\tau_1\lra 0$ and $\tau_2\lra \infty$ as $n\lra \infty$.

\begin{theorem}\label{LboundS}
	Consider the model \eqref{model}. For any positive constants $\tau_{1}<\tau_{2}$, for both fixed $p$ and $p \lra \infty$ as $n \lra \infty$, 
	\bea
	\inf_{\pi\in\Pi_n} \sup_{\sg_0 \in \mathcal{C}(\tau_{1}, \tau_{2})} \Esg \bbE^{\pi} ( \|\sg_n - \sg_0 \|^2 \mid \bfX_n) &\ge& c \cdot \tau_2^2\cdot \min\left(\frac{p}{n}, 1 \right)
	\eea
	for all sufficiently large $n$ and some constant $c>0$.
\end{theorem} 

\begin{theorem}\label{UboundS}
	Consider the model \eqref{model} and prior \eqref{mixWandI} with $\nu_n^2 = O(np)$ and $\|A_n\|^2 = O(np)$. For any positive constants $\tau_1<\tau_2$,
	$$\sup_{\sg_0 \in \cC(\tau_1 , \tau_2)} \Esg \bbE^{\pi} ( \|\sg_n - \sg_0 \|^2 \mid \bfX_n) \le c \cdot \min\left(\frac{p}{n}, 1 \right)$$
	for all sufficiently large $n$ and some constant $c>0$.
\end{theorem}

\subsection{Proof of Theorem \ref{LboundS}}
Lemma \ref{lemma1}-\ref{constINEQ} are used to prove Theorem \ref{LboundS}. The proofs of Lemma \ref{lemma1} and Lemma \ref{lemma2} are straightforward, and they are omitted here. 
\begin{lemma}\label{lemma1}
	Let $f_i$ be the density function of  $p$-dimensional $N_p(0,\sg_i), i=0,1,2$. If $\sg_1^{-1}+ \sg_2^{-2} - \sg_0^{-1}$ is a positive definite matrix,
	\bea
	\int_{\bbR^p} \frac{f_1 f_2}{f_0} dx &=& [\text{det}(I_p - \sg_0^{-2}(\Sigma_1 - \sg_0)(\Sigma_2 - \sg_0) )]^{-1/2}.
	\eea
\end{lemma}

\begin{lemma}\label{lemma2} 
	Define $\cU := \{ u \in \mathbb{R}^p : u_i = \pm 1/\sqrt{p} , i=1, \cdots , p \}$. For any $u, v \sim Unif(\cU)$,
	$$\langle u, v \rangle \stackrel{d}{\equiv} 2 B/p - 1$$
	where $B \sim Bin(p, 1/2)$.
\end{lemma}

\begin{lemma}\label{constINEQ}
	Let $P_0, P_1\in \cP$ where $\cP$ is a set of all probability measures on $\cX$   and let $f_0$ and $f_1$ be their density functions, respectively.
	Define $\xi=\xi(P_0, P_1) := \int_\cX f_1^2/ f_0 dx $ and set $\theta_i=\theta(P_i), i=0,1$, where $\theta$ is a functional defined on $\cP$. Then
	$$\inf_\delta \max\{\bbE_0(\delta-\theta_0)^2, \bbE_1(\delta-\theta_1)^2\}\ge\frac{(\theta_1-\theta_0)^2}{(1+\xi^{1/2})^2},$$
	where $\delta$ denotes any estimator of $\theta$ and $\bbE_i$ represents the expectation with respect to $P_i$, $i=0,1$. 
\end{lemma}
\begin{proof}
	For given estimator $\delta$ which satisfies $R(\delta , \theta_0) = \bbE|\delta(X)-\theta|^2 \leq \epsilon^2$, we have
	$$R(\delta, \theta_1) \geq (|\theta_1 - \theta_0| - \epsilon \xi^{1/2})^2 $$
	by \cite{brown96}.
	Choose $\epsilon = |\theta_1 - \theta_0|/(1+\xi^{1/2})$ so that
	\begin{eqnarray*}
		\epsilon^2 &=& (|\theta_1 - \theta_0| - \epsilon \xi^{1/2})^2 .
	\end{eqnarray*}
	If $\bbE_0 (\delta - \theta_0)^2 \leq \epsilon^2$, we have
	\begin{eqnarray*}
		&& \max \{\bbE_0 (\delta - \theta_0)^2 , \bbE_1 (\delta - \theta_1)^2\}\ge  \bbE_1 (\delta - \theta_1)^2\ge \epsilon^2 = \frac{(\theta_1 - \theta_0)^2}{(1+\xi^{1/2})^2}
	\end{eqnarray*}
	If $\bbE_0 (\delta - \theta_0)^2 \geq \epsilon^2$, we have
	\begin{eqnarray*}
		&&\max\{ \bbE_0 (\delta - \theta_0)^2 ,\bbE_1 (\delta - \theta_1)^2 \}\geq \bbE_0 (\delta - \theta_0)^2\ge \epsilon^2 = \frac{(\theta_1 - \theta_0)^2}{(1+\xi^{1/2})^2}.
	\end{eqnarray*}
	Hence,
	\bea
	\inf_{\delta}\max \{\bbE_0 (\delta - \theta_0)^2 , \bbE_1 (\delta - \theta_1)^2\}
	&\ge& \frac{(\theta_1-\theta_0)^2}{(1+\xi^{\frac{1}{2}})^2}.
	\eea
\end{proof}

\begin{proof}[Proof of Theorem \ref{LboundS}] 
	It suffices to show that 
	\bea
	\inf_{\hat \sg_n} \sup_{\sg_0 \in \cC(\tau_1 , \tau_2)} \Esg \| \hat{\Sigma}_n - \Sigma_0 \|^2 &\ge&  c' \cdot \tau_2^2 \cdot \min\left(\frac{p}{n}, 1 \right)
	\eea
	for some constant  $c'>0$ because by the Jensen's inequality, 
	\bea
	\inf_{\pi\in\Pi_n} \sup_{\sg_0 \in \cC(\tau_1 , \tau_2)} \Esg \bbE^{\pi} (\| \sg_n - \sg_0 \|^2 \mid \bfX_n) 
	&\ge& \inf_{\pi\in\Pi_n} \sup_{\sg_0 \in \cC(\tau_1 , \tau_2)} \Esg \| \tilde{\sg}_n - \sg_0 \|^2\\
	&\ge& \inf_{\hat{\sg}_n} \sup_{\sg_0 \in \cC(\tau_1 , \tau_2)} \Esg \| \hat{\sg}_n - \sg_0 \|^2 ,
	\eea
	where $\tilde{\sg}_n := \bbE^\pi(\sg_n \mid \bfX_n)$.
	Assume $n\ge p$ and define 
	\bea
	\cU &:=& \Big\{ u = (u_1, \ldots, u_p) \in \mathbb{R}^p : u_i = \pm 1/\sqrt{p} , i=1, \cdots , p \Big\}, \\
	\Theta &:=& \Big\{ \Sigma \in \bbR^{p\times p} : \Sigma = \frac{\tau_2}{1+\epsilon} \big[ I_p + \epsilon u u^T\big] , u \in \cU \Big\}
	\eea
	with $\epsilon = c \sqrt{p/n} \le 1$ for some small $c>0$ satisfying $\epsilon \le \tau_2/\tau_1 -1$. 
	Let $P_0^n = N(0, 2^{-1}\tau_2 I_p)^n, P_1^n = 2^{-p} \sum_{\Sigma \in \Theta} N(0, \sg)^n$ and let $f_0^n$ and $f_1^n$ be their density functions, respectively. 
	Note that $\|\sg\| = \tau_2$ and $\lambda_{min}(\sg) = (1+\epsilon)^{-1}\tau_2 \ge \tau_1$ for any $\sg \in \Theta$, thus $\Theta \subset \cC(\tau_1 , \tau_2)$ for some small $c>0$. 
	By the above Lemma \ref{constINEQ},
	\begin{eqnarray*}
		\inf_{\hat \sg_n} \sup_{\sg_0\in \cC(\tau_1 , \tau_2)} \Esg \| \hat{\Sigma}_n - \Sigma_0 \|^2 
		&\geq& \inf_{\hat \sg_n} \sup_{\sg_0\in \cC(\tau_1 , \tau_2)} \Esg( \|\hat{\Sigma}_n\| -\|\Sigma_0\|)^2 \\
		&\geq& \inf_{\delta} \sup_{\sg_0\in \cC(\tau_1 , \tau_2)} \Esg( \delta -\|\Sigma_0\|)^2 \\
		&\geq& \underset{\delta}{\inf} \underset{\Sigma_0 \in \{ 2^{-1}\tau_2 I_p\} \cup \Theta}{\max} \bbE_{\sg_0} ( \delta -\|\Sigma_0\|)^2 \\
		&\ge& \underset{\delta}{\inf} \max (\bbE_{f_0^n} ( \delta - (1+\epsilon)^{-1}\tau_2  )^2, \bbE_{f_1^n} ( \delta - \tau_2 )^2) \\
		&\geq&  \frac{\tau_2^2 \epsilon^2}{4(1+\xi^{1/2})^2},
	\end{eqnarray*}
	where $\delta$ denotes any estimator of $\|\sg_0\|$ and $\xi := \int (f_1^n)^2/ f_0^n $.
	The fourth inequality follows from 
	\bea
	\inf_\delta \max_{\sg \in \Theta} \bbE_{f_\sg}(\delta-\|\sg\|)^2 
	&=& \inf_\delta \max_{\sg\in\Theta} \int (\delta(x) - \tau_2 )^2 f_\sg^n(x) dx  \\
	&\ge& \inf_\delta \frac{1}{2^p} \sum_{\sg\in\Theta} \int (\delta(x)- \tau_2 )^2 f_\sg^n(x) dx\\
	&=& \inf_\delta \int(\delta(x) - \tau_2)^2 f_1^n(x) dx \\
	&=& \inf_\delta \bbE_{f_1^n} (\delta- \tau_2)^2
	\eea	
	where $f_\Sigma^n$ is the density function of $N(0, \Sigma)^n$.
	Now we calculate $\xi$.
	\begin{eqnarray*}
		\xi &=& \int \frac{(f_1^n)^2}{f_0^n}\\ 
		&=& \int \frac{(2^{-p} \sum_{\Sigma \in \Theta} f_{\Sigma}^n)^2}{f_0^n} \\
		&=& \frac{1}{2^{2p}} \sum_{\Sigma_1, \Sigma_2 \in \Theta} \int \frac{f_{\sg_1}^n f_{\sg_2}^n}{f_0^n} \\
		&=& \frac{1}{2^{2p}} \sum_{\Sigma_1, \Sigma_2 \in \Theta} \left(\int \frac{f_{\sg_1} f_{\sg_2}}{f_0} \right)^n\\
		&=&  \frac{1}{2^{2p}} \sum_{u, v \in \cU} \det[(I_p - \epsilon^2 u u^T v v^T)]^{-n/2} \\
		&=& \frac{1}{2^{2p}} \sum_{u, v \in \cU} (1 - \epsilon^2 (u^T v)^2)^{-n/2}  \\
		&=& \bbE(1- \epsilon^2 \langle u,v \rangle^2)^{-n/2} \\
		&\leq& \bbE ( \exp (2n  \epsilon^2 \langle u, v\rangle^2)),
	\end{eqnarray*}
	where $u, v \sim Unif(\cU)$. The fifth equality is derived from Lemma \ref{lemma1}. 
	We will show that $\xi \le C$ for some constant $C>0$ for all sufficiently large $n$.
	If $p$ does not grow to infinity, i.e., $p \le C$ for some constant $C>0$, the last term bounded above easily, $\bbE ( \exp (2n \epsilon^2 \langle u, v\rangle^2)) \le \exp(2 c^2 p) \le \exp(2 c^2 C)$.
	If $p$ tends to infinity, by the Lemma \ref{lemma2}, note that $\sqrt{p}\langle u,v \rangle  \stackrel{d}{\equiv} \sqrt{p}(2B/p -1) \overset{d}{\lra} N(0,1)$ as $n\lra\infty$ where $B \sim Bin(p,1/2)$.
	Note also that we have 
	\bea
	\bbE \left(\exp \left[cp \left(\frac{2}{p}B -1\right)^2 \right] \right) &\lra& \bbE(\exp(cZ^2)) = \frac{1}{\sqrt{1-2c}}
	\eea
	by Theorem 1 of \cite{kozakiewicz1947convergence} for $0 < c < 1/2, ~Z \sim N(0, 1)$. In our setting, consider $F_p$ as the distribution function of $ p (2B/p -1)^2$. 
	Thus, we get the followings by taking $\epsilon = c \sqrt{p/n}$ for some small $c>0$ such that $2c < 1/2$, 
	\begin{eqnarray*}
		\xi &\leq& \bbE \left(\exp \left[2c p \left(\frac{2}{p} B -1\right)^2 \right]\right)\\
		& \lra&  \frac{1}{\sqrt{1-4 c}},
	\end{eqnarray*}
	as $n \lra \infty$.
	Hence, we have
	$$ \underset{\hat{\Sigma}_n}{\inf} \underset{\sg_0\in \cC(\tau_1, \tau_2)}{\sup} \Esg \|\hat{\Sigma}_n - \Sigma_0 \|^2 \geq \frac{\tau_2^2 \epsilon^2}{4(1+\xi^{1/2})^2} \geq c' \cdot \tau_2^2 \cdot \frac{p}{n}$$
	for some $c'>0$ which proves the lower bound when $n\ge p$.
	
	Now, assume $n<p$ and define 
	\bea
	\mathcal{U}_n &:=& \Big\{ u \in \bbR^n : u_i = \pm \frac{1}{\sqrt{n}} , i =1,\ldots, n \Big\} \\
	\Theta &:=& \left\{ \sg = \begin{pmatrix}
		\sg_n & 0 \\
		0 & I_{p-n}
	\end{pmatrix} : \sg_n = \frac{\tau_2}{1+\epsilon} \big[ I_n + \epsilon u u^T\big] , u \in \mathcal{U}_n \right\}.
	\eea
	Earlier result shows that
	\bea
	\inf_{\hat \sg_n} \sup_{\sg_0\in \cC(\tau_1, \tau_2)} \bbE \| \hat{\Sigma}_n - \Sigma_0 \|^2 
	&\ge& \inf_{\hat \sg_n} \max_{\sg_0 \in \Theta} \bbE \| \hat{\Sigma}_n - \Sigma_0 \|^2 \\
	&\ge& c' \cdot \tau_2^2 \cdot \frac{n}{n} = c' \cdot \tau_2^2
	\eea
	for some $c'>0$.
\end{proof}

\subsection{Proof of Theorem \ref{UboundS}}
\begin{lemma}\label{IW_comp_bound} 
	Let $\Omega_n \sim W_p(\nu_n, \nu_n^{-1} A_n)$ with $\nu_n > p$ and positive definite matrix $A_n$, for all $n \geq 1$ and $\|A_n\| \le \tau_n$ for all sufficiently large $n$. Then, there exist positive constants $c_1$ and $c_2$ such that
	\bea
	\bbP( \|\Omega_n - A_n \| \ge x ) &\le& 5^p \left(  e^{-c_1 \nu_n x^2/\tau_n^2 } + e^{- c_2 \nu_n  x/\tau_n } \right)
	\eea
	for all $x>0$.
\end{lemma}

\begin{proof}
	There exist $v_j$ with $\|v_j\|_2=1$ for $j=1,\ldots, 5^p$, such that
	\bea
	\| A\| &\le& 4 \cdot \sup_{j\le 5^p} |v_j^T A v_j |
	\eea
	for any $p\times p$ symmetric matrix $A$ (Page 2141 of \cite{cai2010optimal}). 
	Thus, we have
	\bea
	\bbP(\| \Omega_n - A_n\| \ge x ) 
	&\le& \bbP (\|A_n\| \| A_n^{-1/2} \Omega_n  A_n^{-1/2} - I_p \| \ge x  ) \\ 
	&\le& \bbP  (\| A_n^{-1/2} \Omega_n A_n^{-1/2} - I_p \| \ge x/\tau_n  ) \\ 
	&\le& \bbP \left( 4 \cdot \sup_{j \le 5^p} | v_j^T (A_n^{-1/2} \Omega_n A_n^{-1/2} - I_p)v_j | \ge x/\tau_n \right)  \\
	&\le& 5^p \sup_{j\le 5^p} \pi \left( | v_j^T (A_n^{-1/2} \Omega_n A_n^{-1/2} - I_p)v_j | \ge x/ (4\tau_n) \right) \\
	&\le& 5^p \left( e^{-c_1 \nu_n x^2/\tau_n^2 } + e^{- c_2 \nu_n  x/\tau_n } \right).
	\eea
	The last inequality follows from Lemma 2.4 and Theorem 3.2 of \cite{saulis1991limit} because $A_n^{-1/2} \Omega_n A_n^{-1/2} \sim W_p(\nu_n, \nu_n^{-1} I_p)$. 
\end{proof}

\begin{lemma}\label{maxminev}
	Let $\Omega_n \sim W_p( \nu_n, \nu_n^{-1}I_p)$ with $c\nu_n \ge p$ for some constant $0<c<1$. Then ,
	\bea
	\pi( \lambda_{max}(\Omega_n)  \ge c_1  ) &\le& 2e^{-\nu_n/2},\\
	\pi( \lambda_{min}(\Omega_n) \le c_2 ) &\le& 2 e^{-\nu_n (1-\sqrt{p/\nu_n})^2/8 }
	\eea
	for any constant $c_1 \ge (2+ \sqrt{p/\nu_n})^2$ and $0< c_2 \le (1-\sqrt{p/\nu_n})^2/4$.
\end{lemma}

\begin{proof}
	It follows from Corollary 5.35 in \cite{eldar2012compressed},
	\bean
	\pi( \lambda_{max}(\Omega_n)^{1/2} \ge 1 + \sqrt{p/\nu_n} + t/\sqrt{\nu_n} ) &\le& 2 e^{-t^2/2}, \label{lammaxtail} \\
	\pi( \lambda_{min}(\Omega_n)^{1/2} \le 1 -\sqrt{p/\nu_n} -t/\sqrt{\nu_n} )  &\le& 2 e^{-t^2/2} \label{lammintail}
	\eean
	for any $t\ge 0$.
	If we choose $t= \sqrt{\nu_n}$ for \eqref{lammaxtail}, it gives the first inequality
	\bea
	\pi(\lambda_{max}(\Omega_n) \ge (2+ \sqrt{p/\nu_n})^2 ) &\le& 2 e^{-\nu_n/2}. 
	\eea
	If we choose $t=\sqrt{\nu_n}(1- \sqrt{p/\nu_n}- (1-\sqrt{p/\nu_n})/2) >0$ for \eqref{lammintail}, it gives the second inequality
	\bea
	\pi(\lambda_{min}(\Omega_n) \le (1-\sqrt{p/\nu_n})^2/4) &\le& 2 e^{-\nu_n(1-\sqrt{p/\nu_n}- (1-\sqrt{p/\nu_n})/2)^2/2 }\\
	&\le& 2e^{-\nu_n(1-\sqrt{p/\nu_n})^2/8}. 
	\eea
\end{proof}

\begin{proof}[Proof of Theorem \ref{UboundS}] \text{}
	We prove the upper bound for $p \le n/2$ case first. 
	Note that
	\bean\label{UboundSeq}
	&& \Esg \bbE^{\pi} (\|\sg_{n} - \sg_0\| \mid \bfX_n)  \nonumber\\ 
	&\le& \Esg \bbE^{\pi} (\| \sg_{n} - \breve{\sg}_n \| \mid \bfX_n ) + \Esg \|\breve{\sg}_n - \sg_0 \| ,
	\eean
	where $\breve{\sg}_n  := (n S_n + A_n)/(n+\nu_n)$.
	Consider the first term of right hand side (RHS) of \eqref{UboundSeq}.
	\bean
	&&\Esg \bbE^{\pi} (\| \sg_{n} - \breve{\sg}_n \| \mid \bfX_n )\nonumber \\
	&=& 
	\Esg \left[ \bbE^{\pi} (\| \sg_{n} - \breve{\sg}_n \| \mid \bfX_n )I(\|\breve{\sg}_n \|\le C_1 \text{ and } \|\breve{\sg}_n^{-1} \| \le C_2)\right] \label{IWexpect1} \\
	&+& \Esg \left[ \bbE^{\pi} (\| \sg_{n} - \breve{\sg}_n \| \mid \bfX_n )I(\|\breve{\sg}_n \| > C_1 \text{ or } \|\breve{\sg}_n^{-1} \| > C_2)\right] \label{IWexpect2}
	\eean
	for any constant $C_1$ and $C_2$. 
	The integrand of \eqref{IWexpect1} is bounded by
	\bea
	&& \bbE^{\pi} (\| \sg_{n} - \breve{\sg}_n \| \mid \bfX_n )I(\|\breve{\sg}_n \|\le C_1 \text{ and } \|\breve{\sg}_n^{-1} \| \le C_2) \\
	&\le& \bbE^{\pi} (\|\sg_{n}\| \| \sg_{n}^{-1} \breve{\sg}_n - I_p \| \mid \bfX_n )I(\|\breve{\sg}_n \|\le C_1 \text{ and } \|\breve{\sg}_n^{-1} \| \le C_2) \\
	&\le& \bbE^{\pi} ( \|\sg_{n}\| \| \breve{\sg}_n^{-1/2} \| \| \breve{\sg}_n^{1/2}\sg_{n}^{-1} \breve{\sg}_n^{1/2} - I_p \| \| \breve{\sg}_n^{1/2} \| \mid \bfX_n )I(\|\breve{\sg}_n \|\le C_1 \text{ and } \|\breve{\sg}_n^{-1} \| \le C_2) \\
	&\le&  \sqrt{C_1 C_2} \cdot \bbE^{\pi} (\|\sg_{n}\| \| \breve{\sg}_n^{1/2}\sg_{n}^{-1} \breve{\sg}_n^{1/2} - I_p \| \mid \bfX_n )I(\|\breve{\sg}_n \|\le C_1 \text{ and } \|\breve{\sg}_n^{-1} \| \le C_2)  \\
	&\le& \sqrt{C_1 C_2} \cdot  \left[\bbE^{\pi} (\|\sg_{n}\|^2 \mid \bfX_n )I(\|\breve{\sg}_n \|\le C_1 \text{ and } \|\breve{\sg}_n^{-1} \| \le C_2) \right]^{1/2} \\
	&\times&   \left[\bbE^{\pi} ( \| \breve{\sg}_n^{1/2}\sg_{n}^{-1} \breve{\sg}_n^{1/2} - I_p \|^2 \mid \bfX_n )\right]^{1/2}.
	\eea
	To show that 
	\bea
	\Esg \left[ \bbE^{\pi} (\| \sg_{n} - \breve{\sg}_n \| \mid \bfX_n )I(\|\breve{\sg}_n \|\le C_1 \text{ and } \|\breve{\sg}_n^{-1} \| \le C_2)\right] &\lesssim& \sqrt{p/n},
	\eea
	it suffices to prove that $\bbE^\pi ( \| \breve{\sg}_n^{1/2}\sg_{n}^{-1} \breve{\sg}_n^{1/2} - I_p \|^2 \mid \bfX_n ) \lesssim p/n$ and $\bbE^{\pi} (\|\sg_{n}\|^2 \mid \bfX_n )I(\|\breve{\sg}_n \|\le C_1 \text{ and } \|\breve{\sg}_n^{-1} \| \le C_2) = O(1)$.
	Note that
	\bean
	&& \bbE^{\pi} ( \| \breve{\sg}_n^{1/2}\sg_{n}^{-1} \breve{\sg}_n^{1/2} - I_p \|^2 \mid \bfX_n ) \nonumber \\
	&\le& \int_x^\infty \pi( \| \breve{\sg}_n^{1/2}\sg_{n}^{-1} \breve{\sg}_n^{1/2} - I_p \|^2 \ge u \mid \bfX_n)  du + x  \nonumber \\
	&\le& \int_x^\infty 5^p \left( e^{-C_3(n+\nu_n)u} + e^{-C_4(n+\nu_n)\sqrt{u}} \right) du + x \nonumber \\
	&\le& \frac{5^p e^{-C_3(n+\nu_n)x}}{C_3(n+\nu_n) } + \frac{5^p\cdot 2\sqrt{x} e^{-C_4(n+\nu_n)\sqrt{x}}}{C_4(n+\nu_n)} + \frac{5^p \cdot 2 e^{-C_4(n+\nu_n)\sqrt{x}}}{C_4^2(n+\nu_n)^2} + x \label{last1}
	\eean
	for any $x>0$ and some positive constants $C_3$ and $C_4$ by Lemma \ref{IW_comp_bound}. 
	If we choose $x = C_5 \cdot p/n$ for some large $C_5>0$, the rate of \eqref{last1} is $p/n$.
	Note that
	\bea
	\bbE^\pi\left(\|\sg_n\|^2 \mid \bfX_n \right) 
	&\le& \bbE^\pi \left( \|\sg_n\|^2 I(\|\sg_n\| > C_6) \mid \bfX_n  \right) + C_6^2 \\
	&\le& \left[\bbE^\pi\left(\|\sg_n\|^4\mid \bfX_n \right) \right]^{1/2}  \left[\pi \left(\|\sg_n\|>C_6  \mid \bfX_n\right) \right]^{1/2} + C_6^2,
	\eea
	by H\"{o}lder's inequality. 
	One can easily show that $\bbE^\pi(\|\sg_n\|^4 \mid \bfX_n)$ is bounded above by $p^5$ up to some constant factor because $\|\sg_n\| \le tr(\sg_n)$ and $\bbE^\pi(tr(\sg_n)^4 \mid \bfX_n )I(\|\breve{\sg}_n \|\le C_1 \text{ and } \|\breve{\sg}_n^{-1} \| \le C_2) \lesssim p^4$. Also note that
	\bea
	\pi\left(\|\sg_n\| > C_6 \mid \bfX_n \right) 
	&=& \pi \left( \lambda_{min}(\sg_n^{-1}) < C_6^{-1} \mid \bfX_n \right) \\
	&\le& \pi \left( \lambda_{min}(\breve{\sg}_n^{1/2} \sg_n^{-1} \breve{\sg}_n^{1/2}) < \|\breve{\sg}_n\| C_6^{-1} \mid \bfX_n  )  \right)  \\
	&\le&  \pi \left( \lambda_{min}(\breve{\sg}_n^{1/2} \sg_n^{-1} \breve{\sg}_n^{1/2}) < C_1 C_6^{-1} \mid \bfX_n  )  \right) \\
	&\le& 2 e^{-(n+\nu_n)(1- \sqrt{p/(n+\nu_n)} )^2/8 }
	\eea
	for some constant $C_6 \ge C_1 \cdot 4 (1- \sqrt{1/2})^{-2}$ by Lemma \ref{maxminev}. 
	Thus, we have shown that the rate of \eqref{IWexpect1} is smaller than $\sqrt{p/n}$.
	
	Now, we show that the rate of \eqref{IWexpect2} is smaller than $\sqrt{p/n}$. Note that \eqref{IWexpect2} is bounded by
	\bea
	&& \Esg \left[ \bbE^{\pi} (\| \sg_{n} - \breve{\sg}_n \| \mid \bfX_n )I(\|\breve{\sg}_n \| > C_1 \text{ or } \|\breve{\sg}_n^{-1} \| > C_2) \right]\\
	&\le& \Esg \left[ \left( \bbE^{\pi}(\|\sg_{n}\| \mid \bfX_n) + \|\breve{\sg}_n \| \right) I(\|\breve{\sg}_n \| > C_1 \text{ or } \|\breve{\sg}_n^{-1} \| > C_2) \right] \\
	&\le& \Esg\left[ \bbE^\pi(\|\sg_n\| \mid \bfX_n) I(\|\breve{\sg}_n\| > C_1) \right]  + \Esg\left[ \bbE^\pi(\|\sg_n\| \mid \bfX_n) I(\|\breve{\sg}_n^{-1}\| > C_2) \right] \\
	&+& \Esg \left[ \|\breve{\sg}_n\|  I(\|\breve{\sg}_n \| > C_1) \right] + \Esg \left[ \|\breve{\sg}_n\|  I(\|\breve{\sg}_n^{-1} \| > C_2) \right]
	\eea
	Since $\breve{\sg}_n = (nS_n +A_n)/(n+\nu_n)$ and $\sg_0 \in \cC(\tau_1 , \tau_2)$, we have 
	\bean
	\bbP_{\sg_0} (\|\breve{\sg}_n \| > C_1)
	&\le& \bbP_{\sg_0} \left( \|S_n\| + \frac{\|A_n\|}{n+\nu_n}  > C_1 \right) \nonumber \\
	&=& \bbP_{\sg_0} \left( \|S_n\|  > C_1 - \frac{\|A_n\|}{n+\nu_n} \right) \nonumber \\
	&\le& \bbP_{\sg_0} \left( \| \bar{S}_n \|  >  \tau^{-1}_2 \left(C_1 - \frac{\|A_n\|}{n+\nu_n}\right) \right) \label{barSnprob}
	\eean
	where $\bar{S}_n := \sg_0^{-1/2} S_n \sg_0^{-1/2} \sim W_p(n, n^{-1}I_p)$. 
	Then, \eqref{barSnprob} is bounded by $2 e^{-n/2}$ for some constant $C_1>0$ by Lemma \ref{maxminev}.
	Similarly, for some constant $C_2$,
	\bea
	\bbP_{\sg_0} (\|\breve{\sg}_n^{-1} \| > C_2) 
	&\le& \bbP_{\sg_0} \left( \frac{n+\nu_n}{n} \cdot \|S_n^{-1} \| > C_2 \right) \\
	&=& \bbP_{\sg_0} \left(  \lambda_{min}(S_n)  < \left(1+ \frac{\nu_n}{n} \right)  C_2^{-1} \right) \\
	&\le& \bbP_{\sg_0} \left(  \lambda_{min}(\bar{S}_n)  < \tau_1^{-1}  \left(1+ \frac{\nu_n}{n} \right)  C_2^{-1} \right) \\
	&\le& 2 e^{-n (1- \sqrt{p/n})^2/8},
	\eea
	by Lemma \ref{maxminev}.
	It is easy to show that
	\bea
	\bbE^\pi(\|\sg_n\| \mid \bfX_n) 	
	&\le& \frac{(n+\nu_n)p}{n+\nu_n -p-1}\|\breve{\sg}_n \| 
	\eea
	and 
	\bea
	\Esg \left[ \|\breve{\sg}_n\|  I(\|\breve{\sg}_n \| > C_1) \right]
	&=& \int_0^\infty \bbP_{\sg_0} \left[ \|\breve{\sg}_n\|  I(\|\breve{\sg}_n \| > C_1) \ge u \right] du \\
	&=&  \int_{C_1}^\infty \bbP_{\sg_0}(\| \breve{\sg}_n\| \ge u) du \\
	&\le&  \int_{C_1}^\infty \bbP_{\sg_0}\left( \|\bar{S}_n\| \ge  \tau^{-1}_2\left( u - \frac{\|A_n\|}{n+\nu_n} \right) \right) du  .
	\eea
	By applying $t=\sqrt{n}(\sqrt{\tau^{-1}_2(u - \|A_n\|/(n+\nu_n))} - 1 -\sqrt{p/n})$ to the tail inequality \eqref{lammaxtail}, we have
	\bea
	&&\int_{C_1}^\infty \bbP_{\sg_0}\left( \|\bar{S}_n\| \ge  \tau^{-1}_2\left( u - \frac{\|A_n\|}{n+\nu_n} \right) \right) du  \\
	&\le& \int_{C_1}^\infty  2 e^{-n(\sqrt{\tau^{-1}_2(u - \|A_n\|/(n+\nu_n))} - 1 -\sqrt{p/n})^2/2}   du  \\
	&\le& \int_{C_1}^\infty  2 e^{-n \sqrt{u} C_7/2}   du \\
	&\le& \frac{ \sqrt{C_1}}{C_7n} e^{ - \sqrt{C_1}C_7n/2} + \frac{1}{2C_7^2 n^2} e^{-\sqrt{C_1}C_7 n/2}
	\eea
	for some constant $C_7>0$. 
	Also note that
	\bea
	\Esg \left[ \|\breve{\sg}_n\|  I(\|\breve{\sg}_n^{-1} \| > C_2) \right]
	&\le& \left[ \Esg\|\breve{\sg}_n\|^2 \cdot \bbP_{\sg_0}\left( \|\breve{\sg}_n^{-1}\| > C_2 \right)  \right]^{1/2} \\
	&\le& \left[ \Esg\|\breve{\sg}_n\|^2 \cdot 2 e^{-n (1- \sqrt{p/n})^2/8}  \right]^{1/2}
	\eea
	and 
	\bea
	\Esg\|\breve{\sg}_n\|^2 &\le& 2 \frac{\|A_n\|^2}{(n+\nu_n)^2} + 2\Esg \|S_n\|^2 \\ 
	&\le& 2 \sup_n \frac{\|A_n\|^2}{(n+\nu_n)^2} +\int_0^\infty \bbP_{\sg_0} \left( \|S_n\|^2 \ge u \right) du \\
	&\le&  C_8 + \int_{C_9}^\infty \bbP_{\sg_0} \left( \|\bar{S}_n\| \ge \sqrt{u}/\tau_2 \right) du \\
	&\le& C_8 + \int_{C_9}^\infty 2 e^{-n(u^{1/4}/\sqrt{\tau_2} -1 -\sqrt{p/n} )^2/2} du \\
	&\le& C_8 + \int_{C_9}^\infty 2 e^{-n C_{10}\sqrt{u}/2} du 
	\eea
	for some positive constants $C_8, C_9$ and $C_{10}$ by applying the tail inequality \eqref{lammaxtail}.
	Thus, we have shown that the rate of \eqref{IWexpect2} is faster than $\sqrt{p/n}$.

	For the second term of RHS of \eqref{UboundSeq}, note that
	\bea
	\Esg \|\breve{\sg}_n - \sg_0 \| 
	&\le& \Esg \left\|S_n - \sg_0\right\| + \left(\frac{\nu_n}{n+\nu_n}\right) \Esg \|  S_n \| +  \frac{\|A_n\|}{n+\nu_n} .
	\eea
	
	Since $\nu_n^2 = O(np)$ and $\|A_n\|^2 = O(np)$, it is trivial that $\nu_n/(n+\nu_n) \lesssim \sqrt{p/n}$ and $\|A_n\|/(n+\nu_n) \lesssim \sqrt{p/n}$. 
	One can show that $\Esg \|S_n - \sg_0 \| \le \Esg\|\bar{S}_n - I_p\| \cdot\|\sg_0\| \lesssim \sqrt{p/n}$ by Lemma \ref{IW_comp_bound}. Furthermore, it is easy to prove that $\Esg \|S_n\| \lesssim 1$ because we have proved $\Esg\|\breve{\sg}_n\|^2 \lesssim 1$.
	Thus, we have $\Esg \|\breve{\sg}_n - \sg_0 \|  \lesssim \sqrt{p/n}$.

	For the case $p > n/2$, we have
	\bea
	\Esg \bbE^\pi (\|\sg_n - \sg_0\| \mid \bfX_n) 
	&=&  \| I_p - \sg_0\| \\
	&\le&  \|I_p\| + \|\sg_0\|=  1+ \tau_2
	\eea
	which has the same rate with $\min(p/n,1)$.
\end{proof}

\begin{proof}[Proof of Theorem \ref{UboundS_Pre}]
	It suffices to consider the case $p\le n/2$ because the other part is trivial.
	Note that
	\bean
	\Esg \bbE^\pi \left(\|\sg_n^{-1} - \sg_0^{-1}\| \mid \bfX_n \right) 
	&\le& \Esg\bbE^\pi \left(\|\sg_n^{-1} - \breve{\sg}_n^{-1}\| \mid \bfX_n \right)  + \Esg \| \breve{\sg}_n^{-1} - \sg_0^{-1}\| \nonumber \\
	&=& \Esg\left[\bbE^\pi \left(\|\sg_n^{-1} - \breve{\sg}_n^{-1}\| \mid \bfX_n \right)I(\|\breve{\sg}_n^{-1}\| \le C_1)\right] \label{Spec_Pre1} \\
	&+& \Esg\left[\bbE^\pi \left(\|\sg_n^{-1} - \breve{\sg}_n^{-1}\| \mid \bfX_n \right)I(\|\breve{\sg}_n^{-1}\| > C_1)\right] \label{Spec_Pre2}\\ 
	&+& \Esg \| \breve{\sg}_n^{-1} - \sg_0^{-1}\| \label{Spec_Pre3}.
	\eean
	For the term \eqref{Spec_Pre1}, we have
	\bea
	&& \Esg\left[\bbE^\pi \left(\|\sg_n^{-1} - \breve{\sg}_n^{-1}\| \mid \bfX_n \right)I(\|\breve{\sg}_n^{-1}\| \le C_1)\right] \\
	&\le& C_1 \cdot \Esg \bbE^\pi \left(\| \breve{\sg}_n^{1/2}\sg_n^{-1}\breve{\sg}_n^{1/2} - I_p \| \mid \bfX_n \right) \\
	&\lesssim& \frac{p}{n}
	\eea
	by the argument \eqref{last1} in the proof of Theorem \ref{UboundS}. For the term \eqref{Spec_Pre2}, note that
	\bea
	&& \Esg\left[\bbE^\pi \left(\|\sg_n^{-1} - \breve{\sg}_n^{-1}\| \mid \bfX_n \right)I(\|\breve{\sg}_n^{-1}\| > C_1)\right] \\
	&\le& \Esg \left[ \left( \bbE^\pi \left(\|\sg_n^{-1}\| \mid \bfX_n \right) + \|\breve{\sg}_n^{-1}\| \right) I(\|\breve{\sg}_n^{-1}\| > C_1)\right] \\
	&\le& \Esg \left[ \left( \bbE^\pi \left(\|\breve{\sg}_n^{1/2}\sg_n^{-1}\breve{\sg}_n^{1/2}\| \mid \bfX_n \right) + 1 \right) \|\breve{\sg}_n^{-1}\| I(\|\breve{\sg}_n^{-1}\| > C_1)\right] \\
	&\lesssim& p \cdot \Esg \left[  \|\breve{\sg}_n^{-1}\| I(\|\breve{\sg}_n^{-1}\| > C_1)\right]  \\
	&\le& p \cdot  \left[ \Esg\|\breve{\sg}_n^{-1}\|^2 \right]^{1/2} \cdot \bbP_{\sg_0} \left(\|\breve{\sg}_n^{-1}\| > C_1 \right)^{1/2} \\
	&\lesssim& p^2 \cdot e^{-n(1-\sqrt{p/n})^2/16}
	\eea
	by Lemma \ref{maxminev}. The last term \eqref{Spec_Pre3} is bounded above by
	\bea
	&&\Esg \|\breve{\sg}_n^{-1} - S_n^{-1}\| + \Esg\|S_n^{-1} - \sg_0^{-1}\| .
	\eea
	By the Woodbury formula, it is easy to show that
	\bea
	\Esg \|\breve{\sg}_n^{-1} - S_n^{-1}\| 
	&\le&  \frac{\nu_n}{n}\cdot \Esg \| S_n^{-1}\| + \frac{1}{n^2}\cdot \Esg \| S_n^{-1} (A_n^{-1} + n^{-1}S_n^{-1}) S_n^{-1} \| \\
	&\lesssim& \frac{\nu_n}{n} + \frac{1}{n^2} \left[\Esg\|S_n^{-1}\|^4 \right]^{1/2} \cdot \left[ \Esg\|(A_n^{-1} + n^{-1}S_n^{-1})^{-1}\|^2 \right]^{1/2}  \\
	&\le& \frac{\nu_n}{n} + \frac{1}{n} \left[\Esg\|S_n^{-1}\|^4 \right]^{1/2} \cdot \left[ \Esg\|S_n\|^2 \right]^{1/2} \\
	&\lesssim& \frac{\nu_n}{n}
	\eea
	and
	\bea
	\Esg\|S_n^{-1} - \sg_0^{-1}\|
	&\lesssim& \Esg\left( \|S_n^{-1}\| \|\sg_0^{-1/2}S_n \sg_{0}^{-1/2} - I_p \| \right) \\
	&\le& \left[\Esg \|S_n^{-1}\|^2 \right]^{1/2} \cdot \left[\Esg \|\sg_0^{-1/2}S_n \sg_{0}^{-1/2} - I_p \|^2 \right]^{1/2} \\
	&\lesssim& \sqrt{\frac{p}{n}}
	\eea
	from the arguments used in the proof of Theorem \ref{UboundS}.
\end{proof}

\section{Proof of Theorem \ref{LboundF}}

Before we prove Theorem \ref{LboundF}, we define the total variation affinity and the $L_1$-distance between measures.
\begin{definition}[$L_1$-distance]
	Let $P$ and $Q$ be probability measures with density functions $p$ and $q$ with respect to a $\sigma$-finite measure $\nu$, respectively. 
	Let 
	\bea
	\|P \wedge Q \| &:=& \int p \wedge q \,\, d\nu
	\eea
	be the total variation affinity between $P$ and $Q$, and
	\bea
	\|P- Q\|_1 &:=& \int |p-q| \, d\nu
	\eea
	be the $L_1$-distance between $P$ and $Q$.
\end{definition}

\begin{lemma}[Assouad's Lemma]
	Let the parameter set $\Theta = \{0,1\}^k$, $d$ be a pseudo-metric and $T$ be any estimator of $\psi(\theta)$ based on the observation $X$ from $P_\theta$ with $\theta\in\Theta$. Let $H(\theta, \theta') = \sum_{i=1}^{k}|\theta_i-\theta_{i'}|$. Then for all $s>0$
	$$\max_{\theta\in\Theta}2^s\bbE_\theta d^s(T, \psi(\theta))\ge \min_{H(\theta, \theta')\ge 1}\frac{d^s(\psi(\theta),\psi(\theta'))}{H(\theta, \theta')}\frac{k}{2}\min_{H(\theta, \theta')=1} \|P_\theta\wedge P_{\theta'} \|.$$
\end{lemma}
For the proof of Assouad's lemma, see \cite{assouad83}.

\begin{lemma}\label{logdet}
	For any $p \times p$ symmetric matrix $B$ such that $I_p + tB$ is a positive definite matrix for any $t\in [0,1]$ and $\| B \|_F$ is small,
	\bea
	\log det(I_p + B) &=& tr(B) - R
	\eea
	where $0\le R \le c \| B \|_F^2 $ for some positive constant $c$.
\end{lemma}

\begin{proof}[Proof of lemma \ref{logdet}]
	Using the notation $I_p = ( e_{ij} ) = (e_1 , \ldots , e_p),$ let $ e := vec(I_p) := ( e_1^T , \ldots , e_p^T )^T \in \bbR^{p^2}.$ In the same way, let $b := vec(B) := (b_1^T, \ldots , b_p^T)^T \in \bbR^{p^2}.$ Define a function $h : \bbR^{p^2} \to \bbR$ by
	\bea
	h( vec(A)) &:=& \log det(A) ,
	\eea
	for any $p\times p$ positive definite matrix $A$.
	Then, the Taylor expansion yields 
	\bea
	\log det(I_p+B) \,\,=\,\, h(e+b) &=& h(e) + h'(e)^Tb + \frac{1}{2}b^T h''(e+ tb) b \\
	&=& h'(e)^Tb + \frac{1}{2}b^T h''(e+ tb) b 
	\eea
	for some $t\in [0,1]$, where $|b^T h''(e+ tb) b| \le  \|b\|_2^2 \cdot \| h''(e+ tb) \|$. 
	Note that $\frac{\partial}{\partial A} \log det(A) = (A^{-1})^T$ \cite{petersen2008matrix}, so $h'(a) = vec((A^{-1})^T)$ and
	\bea
	h'(e)^Tb &=& \sumij e_{ji}b_{ij} \\
	&=&  tr( B) . 
	\eea 
	
	We need to prove that $-c \|b\|_2^2 \le b^T h''(e+tb) b/2 \le 0$ for some constant $c>0$. 
	Since $h(a) = \log det(A)$ is concave on positive definite matrices \cite{cover1988determinant}, $h''(a) $ is a negative semidefinite matrix for all positive definite $A$. Thus, $b^T h''(e+tb) b \le 0$. 
	Furthermore, $\|h''(e + tb)\|$ is a continuous function on $t\in [0,1]$ because $I_p + tB$ is a positive definite matrix for any $t\in [0,1]$. Thus, $\| h''(e+ tb)\|/2 \le c$ for some constant $c>0$ uniformly on $t \in [0,1]$. 
\end{proof}

\begin{proof}[Proof of Theorem \ref{LboundF}] \text{}	
	We follow closely the line of a proof in \cite{cai2010optimal}.
	By the Jensen's inequality, 
	\bea
	\inf_{\pi\in\Pi} \sup_{\sg_0 \in \mathcal{C}(\tau)} \Esg \bbE^{\pi} (\| \sg_n - \sg_0 \|_F^2 \mid \bfX_n) 
	&\ge& \inf_{\pi\in\Pi} \sup_{\sg_0 \in \mathcal{C}(\tau)} \Esg \| \tilde{\sg}_n - \sg_0 \|_F^2\\
	&\ge& \inf_{\hat{\sg}_n} \sup_{\sg_0 \in \mathcal{C}(\tau)} \Esg \| \hat{\sg}_n - \sg_0 \|_F^2 \\
	&\ge& \inf_{\hat{\sg}_n} \sup_{\sg_0 \in \mathcal{A}} \Esg \| \hat{\sg}_n - \sg_0 \|_F^2 
	\eea
	for any $\mathcal{A} \subset \mathcal{C}(\tau)$, where $\tilde{\sg}_n = \bbE^\pi ( \sg_n \mid \bfX_n )$.
	We show that for any $\tau>0$ and $\mathcal{A} \subset \mathcal{C}(\tau)$,
	\bea
	\inf_{\hat{\sg}_n} \sup_{\sg_0 \in \mathcal{A}} \Esg \| \hat{\sg}_n - \sg_0 \|_F^2
	&\ge& c \cdot \tau^2 \cdot\frac{p}{n} \cdot \min(p, \sqrt{n})
	\eea
	for some constant $c>0$.
	Note that $\tau$ can depend on $n$ and possibly $\tau\lra \infty$ as $n\lra \infty$.
	
	Without loss of generality, we assume $\tau > 1.$ Define 
	$$\mathcal{A} := \left\{ \sg(\theta) : \sg(\theta) = c_2 \Big[ I_p + \frac{c_1}{\sqrt{n}} \left(\theta_{ij} I(1\le |i - j| < k) \right) \Big] , \theta_{ij} = \theta_{ji} \in \{0, 1 \},\, i,j=1,2,\ldots,p \right\},$$
	where $k = \min(p,\sqrt{n})$, $c_1 = \min (1/3, 1/(3\sqrt{2 c_4}))$ and $c_2 =\tau/(1+c_1)$. The constant $c_4>0$ will be defined later.
	For any $\sg(\theta) \in \mathcal{A},$
	\bea
	\| \sg(\theta) \| &=& \sup_{\|x \| = 1} c_2\, x^T \left( I_p + \frac{c_1}{\sqrt{n}} (\theta_{ij} I(1\le |i-j| < k) \right) x \\
	&=& c_2 + \sup_{\|x \| = 1} c_2 \,x^T \left( \frac{c_1}{\sqrt{n}} (\theta_{ij} I(1\le |i-j| < k) \right) x\\
	&=& c_2 + \left\| \left( \frac{c_1 c_2}{\sqrt{n}} (\theta_{ij} I(1\le |i-j| < k) \right) \right\| \\
	&\le& c_2 + \left\| \left( \frac{c_1 c_2}{\sqrt{n}} (\theta_{ij} I(1\le |i-j| < k) \right) \right\|_1 \\
	&\le& c_2 + \frac{c_1 c_2}{\sqrt{n}} k. 
	\eea 
	By the definition of $k$, $c_1$ and $c_2$, it follows $\| \sg(\theta) \| \le \tau $. Thus, we have $\mathcal{A} \subset \mathcal{C}(\tau)$.
	
	Note that symmetric and diagonally dominant matrix $\sg(\theta) = \left(\sigma_{ij}(\theta) \right),$ i.e., 
	$$\sigma_{ii}(\theta) > \sum_{j \neq i }^p |\sigma_{ij}(\theta)|,$$
	is a positive definite. See, for example, \cite{Harville08}. Also note that
	$$\hspace{1.5cm} \sg(\theta) - \lambda I_p , ~~~{\text{ for all }} 0< \lambda < (1 - 2c_1)c_2$$
	is a diagonally dominant matrix, thus, is positive definite. 
	This implies that the minimum eigenvalue of $\sg(\theta),$ $\lambda_{\min}(\sg(\theta))  > \lambda $ for all $0< \lambda < (1 - 2c_1)c_2$, which in turn, implies 
	$$\lambda_{\min}(\sg(\theta)) \ge (1 - 2c_1)c_2 \ge \frac{c_2}{3} $$
	because $c_1 \le 1/3$.
	Thus,
	$$\| \sg(\theta)^{-1} \| =   \lambda_{\min}(\sg(\theta))^{-1} \le \frac{3}{c_2}.$$
	
	By Assouad's lemma,
	$$
	\inf_{\hat{\sg}_n} \sup_{\sg_0 \in \mathcal{A}} \Esg \| \hat{\sg}_n - \sg_0 \|_F^2 \ge \frac{1}{2^2} \min_{H(\theta , \theta') \ge 1} \frac{\| \sg(\theta) - \sg(\theta') \|_F^2}{H(\theta, \theta' )} \cdot \frac{(2p-k)(k -1)}{4} \cdot \min_{H(\theta , \theta') =1} \| \bbP_{\theta} \wedge \bbP_{\theta'} \|
	$$
	where $H(\theta, \theta') := \sum_{i > j, 1 \leq |i - j | < k}^p | \theta_{ij} - \theta_{ij}' |$. The first factor of the RHS is given by
	\bea
	\min_{H(\theta , \theta') \ge 1} \frac{\| \sg(\theta) - \sg(\theta') \|_F^2}{H(\theta, \theta' )} &=& \min_{H(\theta , \theta') \ge 1} \frac{(\frac{c_1 c_2}{\sqrt{n}})^2 \sum_{1 \leq | i-j | < k } ( \theta_{ij} - \theta_{ij}' )^2}{H(\theta, \theta' )}\\
	&=& \frac{2 c_1^2 c_2^2}{n} \\
	&=& 2 \left(\frac{c_1}{1+c_1} \right)^2 \frac{\tau^2}{n}
	\eea
	because $\theta_{ij}, \theta_{ij}' \in \{0,1 \}$ and $\sum_{1 \leq | i-j | < k } ( \theta_{ij} - \theta_{ij}' )^2 = 2 H(\theta, \theta' )$. The second factor of the RHS is of rate $kp$.
	
	The proof of the theorem will be completed, if we show that $$\displaystyle \liminf_{n \to \infty} \min_{H(\theta , \theta') =1} \| \bbP_{\theta} \wedge \bbP_{\theta'} \| \ge c_3$$
	for some constant $c_3 > 0$. Since $$\| \bbP_{\theta} - \bbP_{\theta'} \|_1 = 2 - 2 \| \bbP_{\theta} \wedge \bbP_{\theta'} \|, $$ 
	it it suffices to prove, when $H(\theta , \theta') =1$,
	$$
	\| \bbP_{\theta} - \bbP_{\theta'} \|_1^2 < 1, \text{ for all sufficiently large } n. 
	$$
	Then, we have $\displaystyle \liminf_{n \to \infty} \min_{H(\theta , \theta') =1} \| \bbP_{\theta} \wedge \bbP_{\theta'} \| > 1/2$.
	Note that by Pinsker's inequality \cite{Csiszar1967information},
	\bean
	\| \bbP_{\theta} - \bbP_{\theta'} \|_1^2 &\le& 2 K(\bbP_{\theta'} , \bbP_{\theta})\nonumber \\
	&=& n \cdot \left[ tr(\sg(\theta') \sg(\theta)^{-1}) - \log det(\sg(\theta') \sg(\theta)^{-1}) - p \right] \label{trlog}
	\eean
	where $K(\bbP_{\theta'} , \bbP_{\theta}) := \int \log(\frac{dP_{\theta'}}{dP_{\theta}}) dP_{\theta'}$ is the Kullback-Leibler divergence. Define $A_1 := \sg(\theta') - \sg(\theta),$ then \eqref{trlog} can be written as
	\bea
	&& n \cdot \left[ tr(A_1 \sg(\theta)^{-1}) - \log det(I_p + A_1 \sg(\theta)^{-1}) \right]\\ 
	&=& n \cdot \left[ tr(\sg(\theta)^{-1/2} A_1 \sg(\theta)^{-1/2}) - \log det(I_p + \sg(\theta)^{-1/2} A_1 \sg(\theta)^{-1/2}) \right] .
	\eea
	
	Consider the diagonalization of $\sg(\theta)^{-1}$, $ \sg(\theta)^{-1} = U V U^T$ where $U$ is an orthogonal matrix and $V$ is a diagonal matrix.
	Since $H(\theta , \theta') =1,$ 
	\bea
	\|\sg(\theta)^{-1/2} A_1 \sg(\theta)^{-1/2} \|_F^2 &=& \| UV^{1/2}U^T A_1 U V^{1/2} U^T \|_F^2 \nonumber \\
	&=& \| V^{1/2} U^T A_1 U V^{1/2} \|_F^2 \nonumber \\
	&\le& \| V \|^2 \|U^T A_1 U \|_F^2 \label{sup}\nonumber \\
	&=& \| \sg(\theta)^{-1} \|^2 \| A_1 \|_F^2 \\
	&\le& \frac{3^2}{c_2^2} \cdot \frac{2c_1^2 c_2^2}{n} \,\,=\,\, 3^2 \cdot \frac{2c_1^2}{n}. \label{min}
	\eea
	Note that $\| \sg(\theta)^{-1/2}A_1 \sg(\theta)^{-1/2}\| \le \|\sg(\theta)^{-1/2}A_1 \sg(\theta)^{-1/2}\|_F^2 \le 3^2\cdot 2 c_1^2/n \le 2/3$ for any $n \ge 3$ because $c_1\le 1/3$. 
	Then $I_p + t\sg(\theta)^{-1/2} A_1 \sg(\theta)^{-1/2}$ is a positive definite matrix for any $t\in [0,1]$ and $\|\sg(\theta)^{-1/2}A_1 \sg(\theta)^{-1/2}\|_F^2$ is small, so we have 
	\bea
	\log det (I_p + \sg(\theta)^{-1/2}A_1 \sg(\theta)^{-1/2}) &=&  tr(\sg(\theta)^{-1/2}A_1 \sg(\theta)^{-1/2} ) - R_n
	\eea
	where $0\le R_n \le c_4 \|\sg(\theta)^{-1/2} A_1 \sg(\theta)^{-1/2}\|_F^2$ for some constant $c_4 >0$ by Lemma \ref{logdet}. Note that the constant $c_4$ does not depend on $c_1$ as long as $c_1 \le 1/3$ and $n\ge 3$.
	Thus, we have
	\bea
	\| \bbP_{\theta} - \bbP_{\theta'} \|_1^2 &\le& n R_n
	\eea
	such that $R_n \le c_4 \| \sg(\theta)^{-1/2}A_1 \sg(\theta)^{-1/2} \|_F^2$ for all large $n$. 
	Since we choose $c_1 = \min (1/3, 1/(3\sqrt{2 c_4}))$, it completes the proof. 
\end{proof}

\section{Proof of Theorem \ref{UboundW}}
\begin{proof}[Proof of Theorem \ref{UboundW}]
	Let $\tilde{\sg}_n = ( \tilde{\sigma}_{n,ij}) := \bbE^{\pi}(\sg_n \mid \bfX_n)$. Note that 
	\bea
	\bbE_{\sg_0}\bbE^{\pi} (\|\sg_n - \sg_0 \|^2_F \mid \bfX_n) 
	&=& \sumij \Esg \bbE^{\pi} \left( (\sigma_{n,ij} - \sigma_{0,ij} )^2 \mid \bfX_n \right) \\
	&=&\sum_{i=1}^p \sum_{j=1}^p \bbE_{\sg_0} \V^{\pi} \big( \sigma_{n,ij}\mid \bfX_n \big) + \sum_{i=1}^p \sum_{j=1}^p \bbE_{\sg_0}\big( \tilde{\sigma}_{n,ij} - \sigma_{0,ij} \big)^2 \\
	&=:& T_1 ~+~ T_2.
	\eea
	
	Let $B_n = (b_{n,ij}) := \sum_{k=1}^nX_k X_k^T + A_n$.  If $n+ \nu_n -p \ge 6$, we have
	\bea
	T_1 &=& \sumij \Esg \Big( ~\frac{(n + \nu_n -p + 1)b_{n,ij}^2 +(n + \nu_n -p -1)b_{n,ii}b_{n,jj}}{(n + \nu_n - p)(n + \nu_n - p -1)^2(n + \nu_n - p -3)} ~\Big)\\
	&\le& \sumij \Esg \Big( ~\frac{2(n + \nu_n -p)b_{n,ii}b_{n,jj}}{(n + \nu_n - p)(n + \nu_n - p -1)^2(n + \nu_n - p -3)} ~\Big)\\
	&\le& \frac{8}{(n + \nu_n - p)^3} \sumij \Esg \Big( b_{n,ii}b_{n,jj} \Big)\\
	&=& \frac{8}{(n + \nu_n - p )^3} \sumij \left( {\text{Cov}}_{\sg_0}(b_{n,ii}, b_{n,jj}) + \Esg b_{n,ii} \cdot \Esg b_{n,jj} \right). 
	\eea 
	
	The remaining steps are given by
	\bea
	T_1 &\le& \frac{8}{(n + \nu_n - p )^3} \sumij  \Big( \sqrt{\V_{\sg_0}(b_{n,ii}) \cdot \V_{\sg_0}(b_{n,jj})} + \Esg b_{n,ii} \cdot \Esg b_{n,jj} \Big)\\
	&=& \frac{8}{(n + \nu_n - p )^3} \sumij  \Big( 2n \sigma_{0,ii} \cdot \sigma_{0,jj} + (n \sigma_{0,ii} + a_{n,ii}) \cdot (n \sigma_{0,jj} + a_{n,jj}) \Big) \\
	&=& \frac{8}{(n + \nu_n - p )^3} \Big( (n^2+2n) \big(\sum_{i=1}^p \sigma_{0,ii} \big)^2 + 2 n \sum_{i=1}^p \sigma_{0,ii} \sum_{j=1}^p a_{n,jj} + \big( \sum_{i=1}^p a_{n,ii} \big)^2 \Big)\\
	&\leq& \frac{8}{(n + \nu_n - p )^3} \Big( (n^2+2n) p^2 \|\sg_0\|^2 + 2np^2 \|\sg_0\| \cdot \|A_n\|  + p^2 \|A_n\|^2 \Big).
	\eea
	Since $\Sigma_0 \in \calC(\tau)$, we have the upper bound for $T_1$,
	\bea
	T_1 &\le& \frac{8}{(n + \nu_n - p )^3} \Big( (n^2 +2n)\tau^2 p^2 +  2n p^2 \tau \|A_n\| + p^2 \|A_n\|^2 \Big) .
	\eea
	
	Similar to $T_1$, we can compute the $T_2$ part by
	\bea\label{T2}
	T_2 &=& \sumij \Esg \Big( \frac{b_{ij}}{n + \nu_n - p - 1} - \sigma_{0,ij} \Big)^2 \nonumber \\
	&=& \sumij \Big(~ \V_{\sg_0}( \frac{b_{n,ij}}{n + \nu_n - p - 1}) + \Big[ \Esg ( \frac{b_{n,ij}}{n + \nu_n - p - 1} - \sigma_{0,ij} ) \Big]^2 ~ \Big) \nonumber \\
	&=& \sumij \Big( ~\frac{n ( \sigma_{0,ij}^2 + \sigma_{0,ii} \sigma_{0,jj})}{(n + \nu_n - p - 1)^2} + \Big[ \frac{(-\nu_n + p + 1)\sigma_{0,ij} + a_{n,ij}}{n + \nu_n - p - 1} \Big]^2 ~ \Big) \nonumber \\
	&\le& \frac{2n}{(n + \nu_n - p - 1)^2} \sumij (\sigma_{0,ii} \sigma_{0,jj}) \\
	&+&  \frac{2}{(n + \nu_n - p - 1)^2} \sumij \Big( (\nu_n - p - 1)^2 \sigma_{0,ij}^2  + a_{n,ij}^2 \Big).
	\eea
	Since $\| \sg_0 \|_F^2 \le p \| \sg_0 \|^2$,
	\bea
	T_2 &\le& \frac{2}{(n + \nu_n - p - 1 )^2}\Big( n \big( \sum_{i=1}^p \sigma_{0,ii} \big)^2 + (\nu_n - p - 1)^2 \|\sg_0\|_F^2  + \|A_n\|_F^2 \Big) \\
	&\le& \frac{2}{(n + \nu_n - p - 1 )^2} \Big( n p^2 \|\sg_0 \|^2 + (\nu_n - p )^2 p \| \sg_0 \|^2  + p \|A_n\|^2 \Big).
	\eea
	Thus, the upper bound of the rate for $T_2$ is
	\bea
	T_2 &\le& \frac{4}{(n + \nu_n - p )^2} \Big( \tau^2 n p^2 + \tau^2(\nu_n - p )^2 p + p \|A_n\|^2 \Big).
	\eea
	
	We have the upper bound of the rate for the P-loss convergence rate
	\bean\label{max_Fro_final}
	&& \sup_{\sg_0 \in \mathcal{C}(\tau)}\bbE_{\sg_0}\bbE^{\pi} (\|\sg_n - \sg_0 \|^2_F \mid \bfX_n)   \nonumber\\
	&&\begin{split} 
		&\le\,\, \frac{c}{(n + \nu_n - p )^3} \Big( n^2 p^2 \tau^2 +  n p^2 \|A_n\| \tau + p^2 \|A_n\|^2 \Big) \\
		&+\,\, \frac{c}{(n + \nu_n - p )^2}  \Big( n p^2 \tau^2 + (\nu_n - p )^2 p \tau^2  + p \|A_n\|^2 \Big)
	\end{split}
	\eean
	for some constant $c > 0 $. Now, we get the upper bound 
	\bea
	\sup_{\sg_0 \in \mathcal{C}(\tau)}\bbE_{\sg_0}\bbE^{\pi} \left(\|\sg_n - \sg_0 \|^2_F \mid \bfX_n \right) &\le&  c\cdot \tau^2 \cdot \frac{p^2}{n} 
	\eea
	if we assume $\nu_n=p$ and $\|A_n\|^2 = O(n)$.
	
	If we assume $p \le \sqrt{n}$, each term in \eqref{max_Fro_final} should be smaller than $\tau^2 \cdot p^2/n$ to obtain the minimax rate. 
	Under this condition, $\nu_n^2= O(np)$ and $\|A_n\|^2 = O(np)$ is the necessary and sufficient condition to attain the minimax rate $\tau^2 \cdot p^2/n$.
\end{proof}

\section{Proof of Theorem \ref{minimaxB}}
To obtain the minimax posterior rate of the Bregman divergence, we need the following lemma from \cite{cai2012optimal}.
\begin{lemma}\label{BequiF}
	Suppose that the eigenvalues of the real symmetric matrices $X$ and $Y$ lie in $[\tau_1, \tau_2]$ for some constants $0<\tau_1< \tau_2$. 
	Then, there exist positive constants $c_1 < c_2$ depending on $\tau_1$ and $\tau_2$ such that
	$$ c_1 \| X-Y\|_F^2 \le D_\phi (X,Y) \le c_2 \|X-Y \|_F^2$$
	for all $D_\phi \in \calD_\Phi$.
\end{lemma}

\begin{proof}[Proof of Theorem \ref{minimaxB}]
	Let $\breve{\sg}_n := (nS_n+A_n)/(n+\nu_n)$. Then,
	\bean
	&& \Esg \bbE^{\pi^{n,K_1,K_2}} \left( D_\phi (\sg_{n} , \sg_0) \mid \bfX_n \right) \nonumber \\
	&\le& C_1 \cdot \Esg \left[ \bbE^{\pi^{n,K_1,K_2}} \left( \|\sg_{n} -\sg_0 \|_F^2 \mid \bfX_n \right)  I(\breve{\sg}_n \notin \cC(C_2,C_3))  \right] \label{trIW_fro1} \\
	&+& C_1 \cdot \Esg \left[ \bbE^{\pi^{n,K_1,K_2}} \left( \|\sg_{n} -\sg_0 \|_F^2 \mid \bfX_n \right)I( \breve{\sg}_n \in \cC(C_2,C_3))  \right] \label{trIW_fro2}
	\eean
	for some constant $C_1>0$ and any positive constants $C_2<C_3$ by Lemma \ref{BequiF}. Set $C_2 = \tau_1 (1-2\tilde{c})^2$ and $C_3 = \tau_2 (1+2\tilde{c})^2$ for some small constant $\tilde{c}>0$.
	Note that \eqref{trIW_fro1} is bounded by
	\bea
	&&\Esg \left[\bbE^{\pi^{n,K_1,K_2}} \left( \|\sg_{n} -\sg_0 \|_F^2 \mid \bfX_n \right)I(\breve{\sg}_n \notin \cC(C_2,C_3)) \right]\\
	&\le& \Esg \left[\bbE^{\pi^{n,K_1,K_2}} \left( p\cdot \|\sg_{n} -\sg_0 \|^2 \mid \bfX_n \right)I(\breve{\sg}_n \notin \cC(C_2,C_3))\right] \\
	&\le& 2p(K_2^2 + \tau^2) \bbP_{\sg_0}(\breve{\sg}_n \notin \cC(C_2,C_3)).
	\eea
	Since $\sg_0\in \cC(\tau_1 , \tau_2)$, $\bbP_{\sg_0}(\breve{\sg}_n \in \cC(C_2,C_3))$ is bounded below by
	\bean
	&&\bbP_{\sg_0}\left( \left(1+\frac{\nu_n}{n} \right)\frac{C_2}{\tau_1} \le \lambda_{min}(\bar{S}_n) \,\, \& \,\, \lambda_{max}(\bar{S}_n) \le \frac{C_3}{\tau_2}\left(1 - \frac{\|A_n\|}{C_3(n+\nu_n)} \right)   \right)\label{barS_23} 
	\eean
	where $\bar{S}_n \sim W_p(n, n^{-1}I_p)$. By applying Corollary 5.35 in \cite{eldar2012compressed} with $t= \tilde{c}\sqrt{n}$, \eqref{barS_23} is bounded below by $1-2 e^{-\tilde{c}^2 n/2}$ for all sufficiently large $n$. Thus, 
	\bea
	p \cdot \bbP_{\sg_0} \Big(\breve{\sg}_n \notin \cC(C_2,C_3) \Big) &\le& 2p e^{-\tilde{c}^2 n/2} \,\,\ll\,\, \frac{p^2}{n}.
	\eea
	
	Note that the integrand of \eqref{trIW_fro2} is bounded by
	\bea
	&&\bbE^{\pi^{n,K_1,K_2}} \left( \|\sg_{n} -\sg_0 \|_F^2 \mid \bfX_n \right)\\
	&=& \int \|\sg_{n}-\sg_0\|_F^2 \frac{d_{IW_p}(\sg_{n}\mid n+\nu_n, nS_n+A_n ) I(\sg_{n} \in \cC(K_1,K_2)) }{\int_{\cC(K_1,K_2)} d_{IW_p}(\sg_{n}'\mid n+\nu_n, nS_n+A_n ) d\sg_n' } d\sg_{n}  \\
	&\le& \frac{1}{\pi(\sg_n \in \cC(K_1,K_2) \mid \bfX_n) }\cdot \bbE^\pi \left(\|\sg_n -\sg_0\|_F^2 \mid \bfX_n  \right)
	\eea
	where $\pi(\sg_n \mid \bfX_n)$ is a density function of $IW_p(n+\nu_n, nS_n+A_n)$. If we show that $\pi(\sg_n \in \cC(K_1,K_2) \mid \bfX_n)^{-1} I(\breve{\sg}_n \in \cC(C_2,C_3))  \le 2$ for all sufficiently large $n$, the rate of \eqref{trIW_fro2} is $p^2/n$ by Theorem \ref{UboundW}.  Note that
	\bea
	&& \frac{I( \breve{\sg}_n \in \cC(C_2,C_3))}{\pi(\sg_n \in \cC(K_1,K_2) \mid \bfX_n)} \nonumber \\
	&=&  \frac{I(\breve{\sg}_n\in \cC(C_2,C_3))}{\pi( K_1 \le \lambda_{min}(\sg_n) \le \lambda_{max}(\sg_n) \le K_2  \mid \bfX_n)}\nonumber\\
	&\le& \frac{I(\breve{\sg}_n\in \cC(C_2,C_3))}{\pi( K_1 \le \lambda_{min}(\breve{\sg}_n^{-1/2} \sg_n \breve{\sg}_n^{-1/2}) \lambda_{min}(\breve{\sg}_n) \le \lambda_{max}(\breve{\sg}_n^{-1/2} \sg_n \breve{\sg}_n^{-1/2}) \lambda_{max}(\breve{\sg}_n) \le K_2  \mid \bfX_n)} \nonumber\\
	&\le& \frac{1}{\pi( K_2^{-1} C_3 \le \lambda_{min}(\breve{\sg}_n^{1/2} \sg_n^{-1} \breve{\sg}_n^{1/2}) \le \lambda_{max}(\breve{\sg}_n^{1/2} \sg_n^{-1} \breve{\sg}_n^{1/2}) \le K_1^{-1}C_2  \mid \bfX_n)},
	\eea
	where $\breve{\sg}_n^{1/2} \sg_n^{-1} \breve{\sg}_n^{1/2} \mid \bfX_n \sim W_p(n+\nu_n, (n+\nu_n)^{-1}I_p)$.
	Note that if $K_1 < \tau_1$ and $K_2 > \tau_2$, we always can find the small constant $\tilde{c}>0$ satisfying $K_1 \le C_2 (1+2\tilde{c})^{-2} = \tau_1 \{(1-2\tilde{c})/(1+2\tilde{c})\}^2$ and $K_2 \ge C_3 (1-2\tilde{c})^{-2} = \tau_2 \{(1+2\tilde{c})/(1-2\tilde{c})\}^2$.
	Then, by applying Corollary 5.35 in \cite{eldar2012compressed} with $t= \tilde{c}\sqrt{n+\nu_n}$, the last term is bounded above by $(1-2 e^{-\tilde{c}^2(n+\nu_n)/2} )^{-1}$ for all sufficiently large $n$. 
	Since $(1-2 e^{-\tilde{c}^2 (n+\nu_n)/2} )^{-1} \le 2$ for all sufficiently large $n$, it completes the proof. 
\end{proof}

\section{Proof of Theorem \ref{minimaxLD}}
\begin{proof}[Proof of Theorem \ref{minimaxLD}] \text{}
	The minimax lower bound part is given at Theorem 3 of \cite{cai2015law}, so we prove here the upper bound part only. Let $\nu_n^2 = O(n/p)$ and $A_n=O_p$. 
	Note that if $\sg \sim IW_p(\nu , A)$, it implies $\det (A \sg^{-1} ) \overset{d}{\equiv} \prod_{k=0}^{p-1} \chi_{\nu-k}^2$ where $\chi_{\nu-k}^2$'s are independent chi-square random variables with the degree of freedom $\nu-k$ (page 180 of \cite{goodman1963distribution}. Then,
	\bea
	&&\Esg \bbE^{\pi} ( (\log \det \sg_n - \log \det \sg_0)^2 \mid \bfX_n) \\
	&=& \Esg \bbE^{\pi} ( (\log \det nS_n - \log \det \sg_0 - \sum_{k=0}^{p-1} \log \chi_{n+\nu_n-k}^2)^2 \mid \bfX_n).
	\eea
	Define $T_n := \log \det S_n - \tau_{n,p}, \tau_{n,p} := \sum_{k=0}^{p-1} \big( \psi((n-k)/2) - \log (n/2) \big)$ and $\psi(x) := d/dz \log \Gamma(z)|_{z=x} $. Then, we have
	\bean
	&& \Esg \bbE^{\pi} \bigg( (\log \det nS_n - \log \det \sg_0 - \sum_{k=0}^{p-1} \log \chi_{n+\nu_n-k}^2)^2 \mid \bfX_n \bigg)\nonumber \\
	&\le&  2 \cdot \Esg  \big(  T_n - \log \det \sg_0 \big)^2  \label{t1} \\
	&+& 2 \cdot \bbE \bigg( \sum_{k=0}^{p-1} [\psi ( (n-k)/2) + \log 2 - \log \chi_{n+\nu_n-k}^2] \bigg)^2 , \label{t3}
	\eean
	where the last expectation is with respect to the chi-square random variables. 
	
	The first term \eqref{t1} has the upper bound 
	\bean\label{rest1}
	\Esg \bbE^{\pi} \Big( \big(  T_n - \log \det \sg_0 \big)^2 \mid \bfX_n  \Big) 
	&\le& -2 \log \left(1-\frac{p}{n} \right) + \frac{10p}{3n (n-p)} .
	\eean
	by Theorem 2 of \cite{cai2015law}. The RHS of \eqref{rest1} has the asymptotic rate $p/n$ because $p = o(n)$.
	
	Using the facts, $\E (\log \chi_\nu^2) = \psi(\nu/2) + \log 2$ and $\V (\log\chi_\nu^2 ) = \psi'(\nu/2)$, we can separate \eqref{t3} into two parts: 
	\bean
	&& \Esg \bbE^{\pi} \left( \left( \sum_{k=0}^{p-1} \left[\psi \left( \frac{n-k}{2} \right) + \log 2 - \log \chi_{n+\nu_n-k}^2 \right] \right)^2 \mid \bfX_n  \right) \nonumber\\
	&\le& 2 \cdot \V_{\sg_0} \left( \sum_{k=0}^{p-1} \log \chi_{n+\nu_n-k}^2 \right) \label{t3_2_1} \\
	&+& 2 \cdot \left( \sum_{k=0}^{p-1} \left[ \psi \left( \frac{n-k}{2} \right) - \psi \left( \frac{n+\nu_n-k}{2} \right) \right] \right)^2. \label{t3_2_2}
	\eean
	Note that $\psi'(\nu) = \nu^{-1} + \theta (2\nu^2)^{-1} + \theta (6 \nu^3)^{-1}$ for $\nu > 1$ and $0< \theta < 1$ (page 169 of \cite{cai2015law}).
	Applying the above facts to \eqref{t3_2_1}, we can show that
	\bean\label{rest3_1}
	\V_{\sg_0} \left( \sum_{k=0}^{p-1} \log \chi_{n+\nu_n-k}^2 \right)  \nonumber
	&=& \sum_{k=0}^{p-1}\left[ \frac{2}{n+\nu_n -k} + \frac{2\theta}{(n+\nu_n-k)^2} + \frac{4\theta}{3(n+\nu_n-k)^3} \right] \nonumber \\
	&\le& \sum_{k=0}^{p-1} \left[ -2 \log \left( 1 - \frac{1}{n+\nu_n-k} \right)  + \frac{7}{3(n+\nu_n-k)^2} \right] \nonumber \\
	&\le& 
	-2 \log \left( 1- \frac{p}{n+\nu_n}\right) + \frac{7}{3}\cdot \frac{p}{n}  
	\eean
	for $0<\theta<1$. In the second line, we use the inequality $x +\theta x^2 \le -\log(1-x)+ x^2/2$ for $0<x<1$. Note that the RHS of \eqref{rest3_1} has the asymptotic rate $p/n$ if $p = o(n)$.
	For \eqref{t3_2_2}, we use the following property of digamma function, $\psi(x+1) - \psi(x) = x^{-1}$. Thus, we have
	\bean\label{rest3_2}
	\left( \sum_{k=0}^{p-1} \left[ \psi \left( \frac{n-k}{2} \right) - \psi \left( \frac{n+\nu_n-k}{2} \right) \right] \right)^2 &\le& \left( \sum_{k=0}^{p-1} \sum_{x=0}^{\lceil \frac{\nu_n}{2} \rceil -1} \frac{2}{n-k + 2x} \right)^2 \nonumber\\
	&\le& \left( \sum_{k=0}^{p-1} \log \left( 1 + \frac{\nu_n+2}{n-k-2} \right) \right)^2 \nonumber \\
	&\le& \left( p \log \left( 1 + \frac{\nu_n+2}{n-p-2} \right) \right)^2 .
	\eean
	Note that \eqref{rest3_2} has the asymptotic rate $p/n$ if $ \nu_n^2 = O( n/p )$ and $p = o(n)$.
	
	Combining \eqref{rest1}-\eqref{rest3_2}, we have
	\bea
	\Esg \bbE^{\pi} ( (\log \det \sg_n - \log \det \sg_0)^2 \mid \bfX_n) &\le&
	- C_1 \log \left( 1 - \frac{p}{n} \right) \\
	&+&  C_2 \cdot \frac{p}{n} + C_3 \cdot  p^2 \left(\log \left(   1 + \frac{\nu_n+2}{n-p-2} \right) \right)^2 
	\eea
	for all sufficiently large $n$ with $n > p$ and some positive constants $C_1,C_2$ and $C_3$. 
	Since we assume $p = o(n)$ and $\nu_n^2 = O(n/p)$,
	\bea
	\Esg \bbE^{\pi} ( (\log \det \sg_n - \log \det \sg_0)^2 \mid \bfX_n) &\le& c\cdot \frac{p}{n}
	\eea
	for all sufficiently large $n$ and some constant $c>0$.
\end{proof}


\bibliographystyle{plain}
\bibliography{mmcov}

\end{document}